\documentclass[english,british]{article}
\title{Black Box Galois Representations.}
\author{Alejandro Arg\'{a}ez-Garc\'{i}a
        \and
        John Cremona}

\usepackage{amsfonts,amsmath,amsthm,amssymb}
\usepackage[linesnumbered,ruled,noend]{algorithm2e}
\usepackage{float}
\usepackage{graphicx}
\usepackage{enumerate}
\usepackage{hyperref}
\usepackage{tikz}
\usetikzlibrary{positioning,decorations.pathreplacing, matrix,arrows,calc}
\tikzset{
  curvedlink/.style={
    to path={
      let \p1=(\tikztostart.east), \p2=(\tikztotarget.west),
      \n1= {abs(\y2-\y1)/4} in
      (\p1) arc(90:-90:\n1) -- ([yshift=2*\n1]\p2) arc (90:270:\n1)
    },
  }
}

\makeatletter
\def\BState{\State\hskip-\ALG@thistlm}
\makeatother

\theoremstyle{plain}
\newtheorem{theorem}{Theorem}[section]
\newtheorem{lemma}[theorem]{Lemma}
\newtheorem{proposition}[theorem]{Proposition}
\newtheorem{corollary}[theorem]{Corollary}

\theoremstyle{definition} 
\newtheorem{definition}[theorem]{Definition}
\newtheorem{example}{Example}
\newtheorem*{example*}{Example}
\newtheorem{remark}[theorem]{Remark}

\newcommand{\F}{\mathbb{F}}
\newcommand{\Q}{\mathbb{Q}}
\newcommand{\Z}{\mathbb{Z}}

\def\P{{\mathbb{P}}}

\DeclareMathOperator{\tr}{tr}
\DeclareMathOperator{\frob}{Frob}
\newcommand{\p}{\mathfrak{p}}
\newcommand{\frobp}{\frob\p}
\newcommand{\n}{\mathfrak{n}}
\newcommand{\N}{\mathfrak{N}}

\newcommand{\oo}{\mathcal{O}}
\newcommand{\FF}{\mathcal{F}}
\newcommand{\Cl}{\mathcal{C}}
\DeclareMathOperator{\gl}{GL}
\DeclareMathOperator{\gal}{Gal}
\DeclareMathOperator{\ord}{ord}
\DeclareMathOperator{\x}{\mathbf{x}}
\DeclareMathOperator{\y}{\mathbf{y}}
\DeclareMathOperator{\z}{\mathbf{z}}
\DeclareMathOperator{\Aa}{\mathbf{A}}

\DeclareMathOperator{\uu}{\mathbf{u}}
\DeclareMathOperator{\UU}{\mathbf{U}}
\DeclareMathOperator{\vv}{\mathbf{v}}

\DeclareMathOperator{\WW}{\mathbf{W}}

\DeclareMathOperator{\qq}{\mathbf{q}}
\DeclareMathOperator{\ee}{\mathbf{e}}

\DeclareMathOperator{\II}{\mathbf{I}}

\DeclareMathOperator{\0}{\mathbf{0}}

\DeclareMathOperator{\Aut}{Aut}
\DeclareMathOperator{\BT}{BT}
\DeclareMathOperator{\Hom}{Hom}
\DeclareMathOperator{\disc}{disc}
\DeclareMathOperator{\rk}{rk}
\newcommand{\rhobar}{\overline{\rho}}
\newcommand{\Kbar}{\overline{K}}
\newcommand{\ndiv}{\nmid}
\newcommand{\CDT}{\v{C}ebotarev Density Theorem}
\newcommand{\Sage}{{\tt Sage}}
\newcommand{\Magma}{{\sc Magma}}
\newcommand{\Pari}{{\tt Pari/GP}}

\newcommand{\lmfdbnumberfield}[1]{\href{http://www.lmfdb.org/NumberField/#1}{#1}}
\newcommand{\lmfdbbmf}[3]{\href{http://www.lmfdb.org/ModularForm/GL2/ImaginaryQuadratic/#1/#2/#3}{#1-#2-#3}}
\newcommand{\lmfdbec}[3]{\href{http://www.lmfdb.org/EllipticCurve/Q/#1/#2/#3}{#1#2#3}}
\newcommand{\lmfdbeciso}[2]{\href{http://www.lmfdb.org/EllipticCurve/Q/#1/#2}{#1#2}}
\newcommand{\lmfdbecnfiso}[3]{\href{http://www.lmfdb.org/EllipticCurve/#1/#2/#3}{#1-#2-#3}}

\begin{document}

\maketitle

\begin{abstract}
We develop methods to study $2$-dimensional $2$-adic Galois
representations $\rho$ of the absolute Galois group of a number
field~$K$, unramified outside a known finite set of primes $S$ of $K$,
which are presented as \textit{Black Box} representations, where we
only have access to the characteristic polynomials of Frobenius
automorphisms at a finite set of primes.  Using suitable finite test
sets of primes, depending only on $K$ and~$S$, we show how to
determine the determinant $\det\rho$, whether or not $\rho$ is
residually reducible, and further information about the size of the
\textit{isogeny graph} of $\rho$ whose vertices are homothety classes
of stable lattices.  The methods are illustrated with examples for
$K=\Q$, and for $K$ imaginary quadratic, $\rho$ being the
representation attached to a Bianchi modular form.

These results form part of the first author's thesis $\cite{argaez01}$.
\end{abstract}

\section{Introduction}\label{sec:intro}
Let $K$ be a number field. Denote by $\Kbar$ the algebraic closure of $K$
and by $G_{K}=\gal(\Kbar/K)$ the absolute Galois group of $K$.  By an
\emph{$\ell$-adic Galois representation} of $K$ we mean a continuous
representation $\rho\colon G_K\to\Aut(V)$, where $V$ is a finite-dimensional
vector space over $\Q_{\ell}$, which is unramified outside a finite
set of primes of~$K$.  Such representations arise throughout arithmetic
geometry, where typically $V$ is a cohomology space attached to an
algebraic variety.  For example, modularity of elliptic curves over
$K$ can be interpreted as a statement that the $2$-dimensional Galois
representation arising from the action of~$G_K$ on the $\ell$-adic
Tate module of the elliptic curve is equivalent, as a representation,
to a representation attached to a suitable automorphic form over~$K$.
In this $2$-dimensional context and with $\ell=2$, techniques have
been developed by Serre \cite{serre-84}, Faltings, Livn\'e
\cite{livne} and others to establish such an
equivalence using only the characteristic polynomial of $\rho(\sigma)$
for a \emph{finite} number of elements $\sigma\in G_K$.  Here the
ramified set of primes~$S$ is known in advance and the Galois
automorphisms $\sigma$ which are used in the Serre-Faltings-Livn\'e
method have the form $\sigma=\frobp$ where $\p$ is a prime not in~$S$,
so that $\rho$ is unramified at~$\p$.

Motivated by such applications, in this paper we study Galois
representations of $K$ as ``Black Boxes'' where both the base
field~$K$ and the finite ramified set~$S$ are specified in advance,
and the only information we have about~$\rho$ is the characteristic
polynomial of $\rho(\frobp)$ for certain primes~$\p$ not in~$S$; we may
specify these primes, but only finitely many of them.  Using such a
Black Box as an oracle, we seek to give algorithmic answers to
questions such as the following (see the following section for
definitions):
\begin{itemize}
\item Is $\rho$ irreducible?  Is $\rho$ trivial, or does it have
  trivial semisimplification?
\item What is the determinant character of~$\rho$?
\item What is the residual representation~$\rhobar$?  Is it
  irreducible, trivial, or with trivial semisimplification?
\item How many lattices in~$V$ (up to homothety) are stable under
  $\rho$ -- in other words, how large is the isogeny class of~$\rho$?
\end{itemize}

In the case where $\dim V=2$ and $\ell=2$, we give substantial answers
to these questions in the following sections.  In
Section~\ref{sec:background} we recall basic facts about Galois
representations and introduce key ideas and definitions, for arbitrary
finite dimension and arbitrary prime~$\ell$.  From
Section~\ref{sec:characters} on, we restrict to $\ell=2$, first
considering the case of one-dimensional representations (characters);
these are relevant in any dimension since $\det\rho$ is a character.
Although in the applications $\det\rho$ is always a power of the
$\ell$-adic cyclotomic character of~$G_K$, we will not assume this,
and in fact the methods of Section~\ref{sec:characters} may be used to
prove that the determinant of a Black Box Galois representation has
this form.  From Section~\ref{sec:residual} we restrict to
$2$-dimensional $2$-adic representations, starting with the question
of whether the residual representation~$\rhobar$ is or is not
irreducible (over~$\F_2$), and what is its splitting field (see
Section~\ref{sec:background} for definitions); a complete solution is
given for both these questions, which we can express as answering the
question of whether or not the isogeny class of~$\rho$ consists of
only one element.  In Section~\ref{sec:small-large} we consider
further the residually reducible case and determine whether or not the
isogeny class of~$\rho$ contains a representative with trivial
residual representation, or equivalently whether the size of the class
is~$2$ or greater.  In Section~\ref{sectionLargeIsogenyClasswidthat2}
we assume that $\rho$ is trivial modulo~$2^k$ for some $k\ge1$ and
determine the reduction of~$\rho\pmod{2^{k+1}}$ completely, in
particular whether it too is trivial.  Hence, for example, we can
determine~$\rho\pmod4$ when $\rhobar$ is trivial, and also as a final
application, in Section~\ref{sec:triviality} we give a (finite)
criterion for whether $\rho$ has trivial semisimplification.

For each of these tasks we will define a finite set~$T$ of primes
of~$K$, disjoint from $S$, such that the Black Box information about
$\rho(\frobp)$ for $\p\in T$ is sufficient to answer the question
under consideration.  In each case except for the criterion for $\rho$
to have trivial semisimplification, only finite $2$-adic precision is
needed about the determinant and trace of~$\rho(\frobp)$, though we
note that in the applications the $2$-adic representation inside the
Black Box is always part of a compatible family of $\ell$-adic
representations, so that in practice these are rational or algebraic
integers and will be known exactly.

The following theorem summarises our results; we refer to later
sections for the definitions of the sets $T_0$, $T_1$ and~$T_2$ and
for algorithms to compute them.  Here $F_{\p}(t)$ denotes the
characteristic polynomial of~$\rho(\frobp)$ (see (\ref{ecuacionF(X)})
below), for a prime~$\p\notin S$.

\begin{theorem}\label{thm:main}
  Let $K$ be a number field and $S$ a finite set of primes of $K$.
  There exist finite sets of primes $T_0$, $T_1$ and $T_2$, disjoint
  from~$S$, depending only on $K$ and~$S$, such that for any
  $2$-dimensional $2$-adic Galois representation $\rho$ of~$G_K$ which
  is continuous and unramified outside~$S$,
  \begin{enumerate}
    \item the reducibility of the residual representation~$\rhobar$, and
      its splitting field when irreducible, are uniquely determined by
      the values of $F_{\p}(1)\pmod2$, \textit{i.e.}, by the
      traces of $\rhobar(\frobp)$, for $\p\in T_0$;
    \item the determinant character $\det\rho$ is uniquely determined
      by the values of $F_{\p}(0)=\det\rho(\frobp)$ for $\p\in T_1$;
    \item when $\rhobar$ is reducible,
      \begin{itemize}
        \item the existence of an equivalent representation whose
          residual representation is trivial is determined by the
          values of $F_{\p}(1)\pmod4$ for $\p\in T_2$;
        \item if $\rho\pmod{2^k}$ is trivial for some $k\ge1$, the
          reduction $\rho\pmod{2^{k+1}}$ is uniquely determined by the
          values of $F_{\p}(1)\pmod{2^{2k+1}}$ for $\p\in T_2$; in
          particular, there is an equivalent representation which is
          trivial modulo~$2^{k+1}$ if and only if
          $F_{\p}(0)\equiv1\pmod{2^{k+1}}$ and
          $F_{\p}(1)\equiv0\pmod{2^{2k+2}}$ for all~$\p\in T_2$;
        \item $\rho$ has trivial semisimplification if and only if
          $F_{\p}(t)=(t-1)^2$ for all~$\p\in T_2$; that is if and only
          if $\tr\rho(\frobp)=2$ and $\det\rho(\frobp)=1$ for
          all~$\p\in T_2$.
      \end{itemize}
  \end{enumerate}
\end{theorem}

In each section we give examples to illustrate the methods, first from
elliptic curves defined over~$\Q$, and then in the final section, we
give two examples arising from Bianchi modular forms, and elliptic
curves over imaginary quadratic fields.  In the examples we refer to
elliptic curves and Bianchi modular forms using their LMFDB labels (see
\cite{lmfdb}) giving links to the relevant object's home pages at
\url{www.lmfdb.org.}

\subsubsection*{Remarks on complexity}
Although Theorem~\ref{thm:main} only states the existence of sets of
primes with certain properties, we will provide algorithms to compute
these, which we have implemented in order to produce examples (see
below).  It is natural, therefore, to ask about the complexity of
these algorithms.  We will not make a precise statement here: as with
essentially all algorithms in algebraic number theory, our algorithms
are exponential in the size of the input, as they require basic
knowledge of the ground field $K$ such as its rings of integers, class
group and unit group.  Computing these from a polynomial defining $K$
was shown to be exponential by Lenstra in \cite{MR1129315}.  Secondly,
our residual reducibility test requires us to be able to enumerate all
extensions of~$K$ unramified outside~$S$ and with Galois group $C_2$,
$C_3$, or $S_3$.  As this is a standard problem we do not give details
of this here, but note that except for fields of small degree and
discriminant, and small sets of primes~$S$, this is likely to be the
slowest step in the overall algorithm.  Computing the $2$-Selmer group
$K(S,2)$ of a number field~$K$ (see (\ref{definitionKS2}) below) can
be highly non-trivial, even for fields~$K$ of moderate degree and
assuming the Generalised Riemann Hypothesis.  Lastly, even if all the
necessary arithmetic data for $K$ is provided as part of the input,
our algorithms rely on being able to find primes satisfying the
conditions for the sets~$T_i$.  In all cases, there are infinitely
many primes with the desired properties, and below we give the
(positive) Dirichlet density of the sets concerned as an informal
indication of how hard finding the primes will be.  Explicit estimates
exist (at least for $K=\Q$) for how large the smallest primes with the
desired property may be, but in practice, for examples where the
previous steps are possible in reasonable time, we are able to find
these primes easily.  Both the number of primes in the sets $T_i$ and
their size (or norm) are relevant in applying these algorithms, since
in practice the work which the Black Box needs to carry out can be
considerable\footnote{For example, this would be the case for the
  examples in \cite{pacetti}, where the second author provided Hecke
  eigenvalues for certain Bianchi modular forms at primes required by
  the authors of~\cite{pacetti}, which included non-principal primes
  of quite large norm.}.

\subsubsection*{Implementation}
We have implemented all the algorithms described in the paper in
\Sage\ (see \cite{sage}).  The code, some of which will be submitted
for inclusion into a future release of \Sage, is available at
\cite{BBGR2ArXiV}. This includes general-purpose code for computing
the test sets $T_0$, $T_1$ and $T_2$ from a number field~$K$ and a
set~$S$ of primes of~$K$, and also worked examples which reproduce the
examples we give in the text.

\section{Background on Galois representations}\label{sec:background}

Fix once and for all a number field $K$ and a finite set $S$ of primes
of $K$.

\begin{definition}
An $\ell$-adic Galois representation over $K$ is a continuous
homomorphism $\rho\colon G_{K}\rightarrow \Aut(V)\cong
\gl_{2}(\Q_{\ell})$, where $V$ is a finite-dimensional vector space
over $\Q_{\ell}$.  Such a representation is said to be
\emph{unramified outside $S$}, if its restriction to the inertia
subgroup at each $\p\notin S$ is trivial.
\end{definition}
We do not assume that the representation $\rho$ is irreducible.

The condition that $\rho$ is unramified outside $S$ means that for
each $\p\notin S$, it factors through the Galois group~$\gal(L/K)$ of
the maximal extension~$L$ of~$K$ unramified at~$\p$.  Since $L/K$ is
unramified at~$\p$, there is a well-defined conjugacy class of
Frobenius automorphisms at $\p$, denoted $\frob{\p}$, in $\gal(L/K)$,
so that for all $\sigma\in\frob\p$,  the values of $\rho(\sigma)$
are conjugate in~$\Aut(V)$ and hence the characteristic polynomial of
$\rho(\sigma)$ is well-defined.  By abuse of notation, we write
$\rho(\frob\p)$ for $\rho(\sigma)$ for any choice of $\sigma$ in this
class, and denote its characteristic polynomial by~$F_{\p}(t)$.
Moreover, by the \CDT, for every automorphism $\sigma\in G_{K}$ there
are infinitely many~$\p\notin S$ for which
$\rho(\sigma)=\rho(\frob\p)$.

From now on we only consider $2$-dimensional representations.
Choosing a basis for~$V$ we may express each $\rho(\sigma)$ as a
matrix, and hence consider~$\rho$ to be a matrix representation
$G_K\to\gl_2(\Q_\ell)$.  Moreover with different choices of bases we obtain equivalent matrix representations.  For $\sigma\in G_K$ define
$F_\sigma(t)$ to be the characteristic polynomial of~$\rho(\sigma)$,
which is a well-defined monic quadratic polynomial in $\Z_{\ell}[t]$,
and for each prime $\p\notin S$ we set $F_\p=F_{\frobp}$, the
\emph{Frobenius polynomial} at~$\p$, which is also well-defined:
\begin{align}\label{ecuacionF(X)}
    F_{\p}(t) &= \det(\rho(\frobp)-t \II)\nonumber\\
    &= t^{2}-\tr(\rho(\frobp))t + \det(\rho(\frobp)) \in \Z_{\ell}[t].
\end{align}
The fact that these polynomials have integral coefficients follows
from the existence of a stable lattice in~$V$, as we recall below.
The information about the representation~$\rho$ that we assume will be
provided consists of the set~$S$ and the values of
$\det(\rho(\sigma))$ and $\tr(\rho(\sigma))$ for $\sigma=\frob{\p}\in
G_{K}$ and $\p\not\in S$.  We encapsulate this setup as an oracle, or
\emph{Black Box}:

\begin{definition}
An $\ell$-adic \emph{Black Box Galois representation} over $K$ with
respect to~$S$ is an oracle which, on being presented with a
prime~$\p$ of $K$, responds with either ``ramified'' if $\p\in S$, or
with the value of the quadratic Frobenius polynomial $F_{\p}(t)$ in
$\Z_{\ell}[t]$ for $\p\notin S$.
\end{definition}
Equivalently, the Black Box delivers for each~$\p\notin S$ the values of
the \emph{trace} $\tr(\rho(\frobp))\in\Z_{\ell}$ and
the \emph{determinant} $\det(\rho(\frobp))\in\Z_{\ell}^*$.

\subsection{Stable lattices and the Bruhat-Tits tree}
It is well known \cite[p.1]{S1} that continuity of~$\rho$ implies the
existence of at least one \emph{stable lattice} $\Lambda$, i.e., a
free $\Z_{\ell}$-submodule of $V$ of full rank such that
$\rho(\sigma)(\Lambda)\subseteq\Lambda$ for all~$\sigma\in G_K$.  With
respect to a $\Z_{\ell}$-basis of~$\Lambda$, $\rho$ determines an
\emph{integral matrix representation} $\rho_{\Lambda}\colon
G_{K}\rightarrow \gl_{2}(\Z_{\ell})$.  Any lattice homothetic to a
stable lattice is also stable and induces the same integral matrix
representation.  Changing to a different $\Z_{\ell}$-basis
of~$\Lambda$ gives rise to an equivalent integral representation
(conjugate within $\gl_{2}(\Z_{\ell})$).  The existence of a stable
lattice shows that the Frobenius polynomials~$F_{\p}(t)$ have
coefficients in~$\Z_{\ell}$.

If we change to a different stable lattice $\Lambda'\subset V$ which
is not homothetic to~$\Lambda$, however, the integral
representation~$\rho_{\Lambda'}$ we obtain, while rationally
equivalent to~$\rho_{\Lambda}$ (conjugate within
$\gl_{2}(\Q_{\ell})$), is not necessarily integrally equivalent
(conjugate within $\gl_{2}(\Z_{\ell})$).  Integral representations
related in this way (rationally but not necessarily integrally
equivalent) are called \emph{isogenous}.  As we are assuming that the
only information we have about~$\rho$ (for fixed $K$ and~$S$) are the
characteristic polynomials of $\rho(\frobp)$ for primes outside~$S$
provided by the Black Box, we cannot distinguish isogenous integral
representations, but still hope to be able to say something about the
set of all of those isogenous to a given one.

\begin{definition}
The \emph{isogeny class} of $\rho$ is the set of pairs
$(\Lambda,\rho_{\Lambda})$ where $\Lambda$ is a stable lattice and
$\rho_{\Lambda}$ the induced map $G_{K}\to\Aut(\Lambda)$, modulo the
equivalence relation which identifies homothetic lattices.
\end{definition}

For each choice of stable lattice and induced integral representation
we can define its associated residual representation.

\begin{definition}
Let $\rho\colon G_K\to\Aut(V)$ be an $\ell$-adic Galois
representation.  To each stable lattice $\Lambda\subset V$ the
associated \emph{residual representation} $\overline{\rho_{\Lambda}}$
is the composite map
$G_K\to\Aut(\Lambda)\to\Aut(\Lambda\otimes_{\Z_{\ell}}\F_{\ell})$.

In matrix terms, $\overline{\rho}_{\Lambda}\colon
G_K\to\gl_2(\F_{\ell})$ is obtained by composing the integral matrix
representation $\rho_{\Lambda}\colon G_{K}\to\gl_2(\Z_{\ell})$ with
reduction modulo $\ell$.
\end{definition}

We cite the following facts (see \cite[p.3]{S1} for the second one):
\begin{itemize}
  \item $\rho$ is irreducible if and only if the number of stable
    lattices, up to homothety, is finite; that is, if and only if the
    isogeny class of~$\rho$ is finite.
  \item Let $\Lambda$ be any stable lattice.  Then the residual
    representation $\overline{\rho_{\Lambda}}$ is irreducible
    over~$\F_{\ell}$ if and only if $\Lambda$ is the only stable lattice up
    to homothety.  In other words, the residual representation is
    irreducible if and only if the isogeny class consists of a single
    element, in which case there is of course only one residual
    representation up to conjugacy in $\gl_2(\F_{\ell})$.
\end{itemize}
From the second fact we see that either all the residual
representations are reducible, or none of them are; in the latter case
there is only one stable lattice up to homothety anyway.  Thus it
makes sense to describe~$\rho$ as ``residually reducible'' or
``residually irreducible'' respectively.

Recall that the $\ell$-adic Bruhat-Tits tree is the infinite graph
whose vertices are the homothety classes of lattices in
$V\cong\Q_{\ell}^2$, with two vertices joined by an edge if their
classes have representative lattices $\Lambda_1$, $\Lambda_2$ such
that $\Lambda_1$ contains $\Lambda_2$ with index~$\ell$.  (This is a
symmetric relation since then $\Lambda_2$ contains $\ell\Lambda_1$
with index~$\ell$.)  Each vertex has degree exactly~$\ell+1$.
Restricting to lattices which are stable under our
representation~$\rho$, we obtain the following:

\begin{definition}
The \emph{stable Bruhat-Tits tree} or \emph{isogeny graph} of an
$\ell$-adic representation $\rho$ is the full subgraph $\BT(\rho)$ of
the Bruhat-Tits tree whose vertices are stable lattices.
\end{definition}

It is easy to see that if $[\Lambda]$ and $[\Lambda']$ are stable
homothety classes, all vertices in the unique path between them are
also stable: we may choose representatives $\Lambda$, $\Lambda'$ in
their homothety classes such that $\Lambda\subseteq\Lambda'$ and the
quotient~$\Lambda'/\Lambda$ is cyclic, of order $\ell^n$ for
some~$n\ge0$.  Now this quotient has a unique subgroup of each
order~$\ell^k$ for $0\le k\le n$, corresponding to a
lattice~$\Lambda''$ with $\Lambda\subseteq\Lambda''\subseteq\Lambda'$,
and by uniqueness, each such $\Lambda''$ is stable.

Hence the stable Bruhat-Tits tree is indeed
a tree. Its vertex set is the isogeny class of~$\rho$ as defined
above, and we may refer to its edges as \textit{$\ell$-isogenies}.  Given two
adjacent stable lattices,  we may choose bases so that the
associated integral matrix representations are conjugate
within~$\gl_2(\Q_{\ell})$ via the matrix
$\begin{pmatrix}\ell&0\\0&1 \end{pmatrix}$.  In $\BT(\rho)$ it is no
longer the case that every vertex has degree~$\ell+1$; considering the
action of~$\gl_2(\F_\ell)$ on~$\P^1(\F_\ell)$ we see that for $\ell=2$
the possible degrees are~$0$, $1$ and~$3$ while for $\ell\ge3$
the possible degrees are~$0$, $1$, $2$ and~$\ell+1$.

We define the \emph{width} of the isogeny class~$\BT(\rho)$ to be the
length of the longest path in~$\BT(\rho)$; by the facts above, this is
finite if and only if $\rho$ is irreducible, and is positive if and
only if $\rho$ is residually reducible.

\section{Characters and quadratic extensions}\label{sec:characters}

The problem of distinguishing continuous $2$-adic characters
($1$-dimensional representations) $\chi\colon G_K\to\Z_2^*$ reduces to
that of distinguishing quadratic extensions of~$K$, since $\Z_2^*$ is
an abelian pro-$2$-group.  Moreover, the image of $\rho$
in~$\gl_2(\Z_2)$ is itself a pro-$2$-group in the case that the
residual representation is reducible, so the technique we describe in
this section will be used later to study both $\det\rho$ and $\rho$
itself in the residually reducible case.

There are only finitely many quadratic extensions $L$ of $K$
unramified outside $S$; their compositum is the maximal extension
of~$K$ unramified outside $S$ and with Galois group an elementary
abelian $2$-group.  Each has the form $L=K(\sqrt{\Delta})$ for a
unique $\Delta\in K(S,2)\leq K^{*}/(K^{*})^{2}$, where $K(S,2)$ is the
subgroup (often called the $2$-Selmer group of~$K$, or of $K^*$) given
by
\begin{align}\label{definitionKS2}
K(S,2)=\{a\in K^{*}/(K^{*})^{2}:\ord_{\p}(a)\equiv 0\pmod2\text{ for all }\p\not\in S\}.
\end{align}
Moreover, when $S$ contains all primes of~$K$ dividing~$2$,
every extension $K(\sqrt{\Delta})$ with $\Delta\in K(S,2)$ is
unramified outside~$S$.  In general the~$\Delta$ such that
$K(\sqrt{\Delta})$ is unramified outside $S$ form a
subgroup~$K(S,2)_u$ of~$K(S,2)$.  We will call elements of~$K(S,2)_u$
\emph{discriminants}, and always regard two discriminants as equal
when their quotient is a square in~$K^*$.  It is convenient here
to consider $\Delta=1$ as a discriminant, corresponding to the trivial
extension $L=K$.

The group of discriminants~$K(S,2)_u$ is an elementary abelian
$2$-group, of cardinality $2^{r}$ with $r\geq 0$, and may also be
viewed as an $r$-dimensional vector space over $\F_{2}$.
Fixing a basis $\{\Delta_{i}\}_{i=1}^{r}$ for $K(S,2)_u$, we may identify
\begin{align}\label{isomorphismKS2F2}
\F_{2}^r&\leftrightarrow K(S,2)_u\nonumber\\
\x&\leftrightarrow \prod_{i=1}^{r}\Delta_{i}^{x_{i}},
\end{align}
where $\x=(x_{i})_{i=1}^{r}$.

Each prime $\mathfrak{p}\notin S$ determines a linear map
\begin{align*}
\alpha_{\mathfrak{p}}\colon K(S,2)_u\to\F_{2}
\end{align*}
defined by $\alpha_{\p}(\Delta)=[\Delta\mid\mathfrak{p}]$, where we set
\begin{align*}
  [\Delta\mid\mathfrak{p}]&=
  \begin{cases}0\pmod2 & \text{if $\mathfrak{p}$ splits in $K(\sqrt{\Delta})$ or $\Delta=1$}\\
               1\pmod2 &\text{if $\mathfrak{p}$ is inert in $K(\sqrt{\Delta})$}.  \end{cases}
\end{align*}
Linearity follows from the relation
$[\Delta\Delta'|\p]=[\Delta|\p]+[\Delta'|\p]$.

For any prime~$\p\notin S$, we define
\begin{align}\label{definitionI(p)}
I(\p)=\{i: [\Delta_{i}\mid\p]=1\} \subseteq \{1,2,\dots,r\}.
\end{align}
Conversely, for each subset $I\subseteq\{1,...,r\}$, we denote by
$\p_{I}$ any prime such that~$I(\p_I)=I$, so that
\begin{align}\label{propiedaddeladelta}
[\Delta_{i}\mid \p_{I}]=1\Leftrightarrow i\in I.
\end{align}
When $I=\{i\}$ or $I=\{i,j\}$ with $i\not=j$, we simply write
$\p_{i}=\p_{\{i\}}$ and $\p_{ij}=\p_{\{i,j\}}$.  By the \CDT\ applied
to the compositum of the extensions~$K(\sqrt{\Delta})$ for $\Delta\in
K(S,2)_u$, the set of primes of the form~$\p_I$ has density $1/2^{r}$
for each subset~$I$, and in particular is infinite.

Each set of primes of the form~$\{\p_i\mid 1\le i\le r\}$ determines a
basis $\{\alpha_{\p_i}\mid 1\le i\le r\}$ for the dual space
$K(S,2)_u^*=\Hom_{\F_2}(K(S,2)_u,\F_2)$, and may be used to
distinguish between two characters unramified outside $S$.  More
generally we make the following definition.

\begin{definition}\label{deflinearlyindependent}
A set $T_{1}$ of primes of $K$ is \emph{linearly independent with
  respect to} $S$ if $T_1$ is disjoint from~$S$ and the linear
functions $\{\alpha_{\p}\mid\p\in T_{1}\}$ form a basis for the dual
space~$K(S,2)_u^*$.
\end{definition}

As observed above, such a set always exists, for example any set of
the form
\begin{align}\label{specialT1}
\{\p_{1},...,\p_{r}\},
\end{align}
defined above with respect to a basis of~$K(S,2)_u$, is a linearly
independent set of primes. We fix once and for all a linearly
independent set of primes, and denote it by $T_{1}$, and can assume
that $\{\alpha_{\p}\mid\p\in T_1\}$ is a dual basis for the chosen
basis $\{\Delta_i\mid 1\le i\le r\}$ for~$K(S,2)_u$.  In practice this
is most easily done by computing $T_1=\{\p_1,\dots,\p_r\}$ first, and
then taking $\{\Delta_i\}$ to be the basis dual to~$\{\alpha_{\p_i}\}$
(see Algorithm~\ref{algoritmoT1} below).  We then see by
$(\ref{propiedaddeladelta})$ that for all~$I\subseteq\{1,2,\dots,r\}$,
\begin{align}\label{alphaij=alphaialphaj}
\alpha_{\p_{I}}(\Delta)=\sum_{i\in I}\alpha_{\mathfrak{p}_{i}}(\Delta).
\end{align}

\begin{algorithm}[H]\label{algoritmoT1}
\caption{To determine a linearly independent set $T_{1}$ of primes of $K$.} 
\SetKwInOut{Input}{Input}
\SetKwInOut{Output}{Output}
\Input{A number field $K$.\\
       A finite set $S$ of primes of $K$.}
\Output{$T_{1}=\{\p_1,\dots,\p_r\}$, a set of primes of $K$ linearly
  independent with respect to~$S$, and a basis for $K(S,2)_u$ dual to $T_1$.}

Let $\{\Delta_{i}\}_{i=1}^{r}$ be  a basis for $K(S,2)_u$\;
Let $T_{1}=\{\}$\;
Let $A$ be a $0\times r$ matrix over $\F_{2}$\;
\While{$\emph{rank}(A)<r$}
{
Let $\p$ be a prime not in $S\cup T_{1}$\;
Let $\vv=([\Delta_{1}|\p],...,[\Delta_{r}|\p])$\;
\If{$\vv$\emph{ is not in the row-space of }$A$}
	{Let $A=A+\vv$; ~~~~\# i.e., adjoin $\vv$ as a new row of $A$\\
	Let $T_{1}=T_{1}\cup \{\p\}$.
	}
}
Let $\tilde\Delta_j=\prod_i\Delta_i^{b_{ij}}$ for $1\le i\le r$, where
 $(b_{ij})=A^{-1}$\;
\Return $T_{1}$ and $\{\tilde\Delta_i\mid1\le i\le r\}$.
\end{algorithm}

In line~5 of the algorithm, and similarly with later algorithms to
determine other special sets of primes, we systematically consider all
primes of~$K$ in turn, for example in order of norm, omitting those
in~$S$.

In line~10, we adjust the initial basis for $K(S,2)_u$ to one which is
dual to the computed set~$T_1$; this is more efficient than fixing a
basis for $K(S,2)_u$ and looking for primes which form a dual basis.

\begin{remark}
  Finding $K(S,2)$ is implemented in standard software packages.  In
  \Sage, {\tt K.selmergroup(S,2)} returns a basis, while in \Magma\ one
  obtains $K(S,2)$ as abelian group via {\tt pSelmerGroup(2,S)}.
  See \cite{sage} or \cite{magma} respectively.

  Methods for computing $K(S,2)$ are based on the short exact
  sequence
  \[
  1 \to \oo_{K,S}^*/(\oo_{K,S}^*)^2 \to K(S,2) \to \Cl_{K,S}[2] \to 1,
  \]
  where $\oo_{K,S}^*$ is the group of $S$-units and $\Cl_{K,S}[2]$ is
  the $2$-torsion subgroup of the $S$-class group~$\Cl_{K,S}$ of~$K$.
  They therefore rely on being able to compute the unit group and
  class group.
\end{remark}


\subsection{Identifying quadratic extensions}\label{subsec:quadratics}
As an easy example of how to use a set~$T_1$ of primes linearly
independent with respect to~$S$, we may identify any extension $L/K$
known to be of degree at most~$2$ and unramified outside~$S$.
Enumerating $T_1=\{\p_1,\dots,\p_r\}$ and the dual basis
$\{\Delta_1,\dots,\Delta_r\}$ for~$K(S,2)_u$, set
\[
\Delta = \prod_{i=1}^{r}\Delta_i^{[L|\p_i]},
\]
where for~$\p\notin S$ we set $[L|\p]=0$ (respectively~$1$) if $\p$ is
split (respectively, is inert) in~$L$.  Then $L=K(\sqrt{\Delta})$ (or
$L=K$ if $\Delta=1$).  The proof is clear from the fact that $L$ is
uniquely determined by the set of primes which split in $L/K$.  In
particular, $L=K$ if and only if all primes in~$T_1$ split.


\subsection{1-dimensional Galois representations}\label{1dimensionalgal}

We first consider additive quadratic characters~$\alpha\colon
G_K\to\F_2$ which are unramified outside~$S$, and see that a linear
independent set~$T_1$ can determine whether such a character is
trivial, and more generally when two are equal.

\begin{lemma}\label{quadraticchar}
Let $\alpha, \alpha_1, \alpha_2\colon G_{K}\rightarrow\F_{2}$ be
additive quadratic characters unramified outside $S$.
\begin{enumerate}
  \item If $\alpha(\frob{\mathfrak{p}})=0$ for all $\p\in T_{1}$, then
    $\alpha=0$.
  \item $\alpha_1=\alpha_2$ if and only if
    $\alpha_1(\frobp)=\alpha_2(\frobp)$ for all~$\p\in T_1$.
\end{enumerate}
\end{lemma}
\begin{proof}
If $\alpha\neq 0$, then the fixed field of $\ker(\alpha)$ is a quadratic
extension $K(\sqrt{\Delta})$ for some non-trivial $\Delta$ in
$K(S,2)$. But $[\Delta|\mathfrak{p}]=\alpha(\frob{\p})=0$ for all
$\p\in T_{1}$, which implies that $\Delta=1$.  For the second part,
consider $\alpha=\alpha_1-\alpha_2$.
\end{proof}

Now let $\chi\colon G_K\to\Z_2^*$ be a $2$-adic character unramified
outside $S$.  For example we may take $\chi=\det\rho$ where $\rho$ is
a $2$-adic Galois representation unramified outside $S$.  Again, to
show triviality of $\chi$, or equality of two such characters, it is
enough to consider their values on $\frobp$ for $\p\in T_1$.

\begin{theorem}\label{determinante}
Let $\chi, \chi_1, \chi_2\colon G_{K}\rightarrow\Z_{2}^*$ be
continuous characters unramified outside $S$.  Let $T_{1}$ be a
linearly independent set of primes with respect to~$S$.
\begin{enumerate}
  \item If $\chi(\frob{\mathfrak{p}})=1$ for all $\p\in T_{1}$, then
    $\chi$ is trivial.
  \item $\chi_1=\chi_2$ if and only if
    $\chi_1(\frobp)=\chi_2(\frobp)$ for all~$\p\in T_1$.
\end{enumerate}
\end{theorem}
\begin{proof}
As before, the second part follows from the first on considering
$\chi=\chi_1\chi_2^{-1}$, which is again a character $ G_{K}\rightarrow
\Z_{2}^{*}$ unramified outside~$S$.

Suppose that $\chi\neq 1$. Let $k\geq 1$ be the greatest integer such
that $\chi(\sigma)\equiv 1\pmod{2^{k}}$ for all~$\sigma\in G_K$. Note
that $\chi(\sigma)\equiv 1\pmod{2}$ for all $\sigma\in G_{K}$, so $k$
does exist. We can write
$$
\chi(\sigma)\equiv 1+2^{k}\alpha(\sigma)\pmod{2^{k+1}}
$$ where $\sigma\mapsto\alpha(\sigma)$ is a non-trivial additive
quadratic character $G_{K}\rightarrow\F_{2}$. However,
$\alpha(\frob{\mathfrak{p}})\equiv 0\pmod 2$ for all
$\mathfrak{p}\in T_{1}$, since $\chi(\frobp)=1$, so by Lemma
\ref{quadraticchar} we have that $\alpha=0$, contradicting the
minimality of $k$.
\end{proof}


\section{Determining the residual representation}\label{sec:residual}

Given a Black Box Galois representation $\rho$, we would like to
determine whether its residual representations are irreducible or
reducible. Recall that this is a well-defined question, even when
there is more than one stable lattice.  In the irreducible case, we
will moreover determine the (unique) residual representation
completely, both its image (which has order~$3$ or $6$, and is
isomorphic to either $C_3$ (the cyclic group of order~$3$) or $S_3$
(the symmetric group of degree~$3$)), and the fixed field of its
kernel.  Note that $\gl_2(\F_2)\cong S_3$, the isomorphism coming from
the action of~$\gl_2(\F_2)$ on $\P^1(\F_2)$.

This is our initial step in determining the size and structure of the
attached Bruhat-Tits tree $\BT(\rho)$, as we will determine whether it
has only one vertex (and width~$0$) or is larger (positive width).

Fixing one stable lattice~$\Lambda$ with residual representation
$\rhobar_\Lambda$, we define the \emph{splitting field} of
$\rhobar_\Lambda$ to be the fixed field of its kernel.  This is an
extension~$L$ of~$K$ which is unramified outside~$S$ such that
$\gal(L/K)\cong\rhobar_\Lambda(G_K)\le\gl_2(\F_2)$, hence $\gal(L/K)$
is isomorphic to one of: $C_1$ (the trivial group), $C_2$ (cyclic of
order~$2$), $C_3$ or~$S_3$.  The first two cases occur
when~$\rhobar_\Lambda$ is reducible, in which case a different choice
of stable lattice may change the image between being trivial and of
order~$2$, while in the residually irreducible case the image and
kernel are both well-defined.

We now show how to identify the residual splitting field, leaving
until a later section the task of saying more in the reducible case.

\subsection{Identifying cubic extensions}\label{subsec:cubics}
The key to our method is that there are only finitely many Galois
extensions~$L/K$, unramified outside~$S$, and with Galois group
either~$C_3$ or~$S_3$, and we may determine these algorithmically.  We
will not discuss here details of this, except to remark that in the
$S_3$ case we can first construct all possible quadratic extensions
$K(\sqrt{\Delta})$ using $\Delta\in K(S,2)_u$ as in the previous
section, and then use either Kummer Theory or Class Field Theory to
construct all cyclic cubic extensions of $K$ or~$K(\sqrt{\Delta})$.
Full details of the Kummer Theory method, using special cases of
results by Cohen \cite{cohen}, can be found in \cite[\S3]{Koutsianas}
(see also Koutsianas's thesis \cite{aggelos}); we have an
implementation of this method in~\Sage.  An alternate implementation,
using Class Field Theory, was written in \Pari\ by Pacetti, as used in
\cite{pacetti} in the case where $K$ is an imaginary quadratic field.
These implementations were used for the examples below.

For present purposes, we assume that, given $K$ and~$S$, we can write
down a finite set~$\FF$ of irreducible monic cubic polynomials in
$\oo_{K}[x]$, whose splitting fields are the Galois extensions $L/K$
unramified outside $S$ with $\gal(L/K)$ isomorphic to either $S_{3}$
or $C_{3}$. Note that the discriminants of the polynomials in~$\FF$
may be divisible by primes not in~$S$, and these primes will need to
be avoided, so we denote by~$S(\FF)$ the union of~$S$ with all prime
divisors of~$\{\disc(f)\mid f\in\FF\}$.

We can characterise the fields $L$ by examining the splitting
behaviour of primes $\p\not\in S(\FF)$, which depends only on the
factorisation of the respective $f\in\FF$ modulo~$\p$.

\begin{definition}\label{functionlambda}
For a monic cubic polynomial $f\in\oo_{K}[x]$ and prime
$\mathfrak{p}\ndiv\disc f$,
define
$$
   \lambda(f,\mathfrak{p})=\begin{cases} 1& \text{if }f\text{ is
  irreducible mod } \mathfrak{p};\\ 0& \text{otherwise.} \end{cases}
$$
\end{definition}

This definition is motivated by the observation that elements
of~$\gl_2(\F_2)$ have trace~$1$ (respectively, $0$) if their order
is~$3$ (respectively, $1$ or~$2$), combined with the following result
from elementary algebraic number theory.
\begin{lemma}\label{lemmafunctionlambda}
Let $f$ be an irreducible monic cubic polynomial in
$\mathcal{O}_{K}[x]$ with splitting field $L$. Then for $\p\ndiv
\disc{f}$,
$$\lambda(f,\p)=\begin{cases}
1& \text{if }\frob{\p} \text{ has order } 3 \text{ in } \gal(L/K)\\
0& \text{if }\frob{\p} \text{ has order } 1 \text{ or } 2 \text{ in } \gal(L/K).\\  \end{cases}$$
\end{lemma}

Hence, if our Black Box representation~$\rho$ has irreducible residual
representation with residual splitting field defined by the cubic~$f$,
we will have
\[
   \lambda(f,\p) \equiv \tr(\rho(\frob{\p}))\pmod
   2\qquad\forall\p\notin S.
\]
This underlies our algorithm for testing residual irreducibility: see
Proposition~\ref{lemmac3s3}.  To this end, we now define a finite set
of primes which can distinguish between the possible splitting
fields~$L$.

\begin{definition}\label{defT0}
Let $\FF$ be a set of monic cubic polynomials in $\oo_K[x]$ whose
splitting fields are exactly the $S_{3}$ and $C_{3}$ extensions of $K$
unramified outside~$S$. An ordered set of primes
$T_{0}=\{\p_{1},...,\p_{t}\}$ of $K$ is a \emph{distinguishing set}
for $(\FF,S)$ if
\begin{enumerate}[$(1)$]
\item $T_{0}\cap S(\FF)=\emptyset$ (equivalently, $T_{0}\cap
  S=\emptyset$ and ${\p}\ndiv\disc f$ for all $\p\in T_{0}$ and $f\in
  \FF$);
\item the vectors
  $(\lambda(f,\p_{1}),...,\lambda(f,\p_{t}))\in\F_{2}^{t}$ for $f\in
  \FF$ are distinct and non-zero.
\end{enumerate}
\end{definition}
We will write
$\vv(f,T_{0})=(\lambda(f,\p_{1}),...,\lambda(f,\p_{t}))$ when
$T_{0}=\{\p_{1},...,\p_{t}\}$.
\begin{lemma}\label{lemmaexistencia}
A distinguishing set of primes for $(\FF,S)$ exists.
\end{lemma}
\begin{proof}
Let $\FF=\{f_{i}\}_{i=1}^{n}$ be the set of monic cubic polynomials
defining the $S_{3}$ and $C_{3}$ extensions of $K$. Set $f_{0}=x^{3}$
and define $\lambda(f_{0},\p)=0$ for all $\p$. It is enough to show
that for all $0\leq j< i\leq n$ there exists a prime $\p\not\in
S(\FF)$ such that $\lambda(f_{i},\p)\neq \lambda(f_{j},\p)$. For
$i\geq 1$ let $L_{i}$ be the splitting field of $f_{i}$. We divide the
proof into three cases. (For more details of the density calculations,
see \cite[p.~21, Lemma 3.2.5]{argaez01}.)

Case 1: When $j=0$, we require for each~$i\ge1$ the existence
  of a prime~$\p$ such that $\lambda(f_{i},\p)=1$. By the
  \CDT, there are infinitely many such
  primes, with density $\frac{1}{3}$ when $\gal(L_{i}/K)\cong S_{3}$,
  or $\frac{2}{3}$ when $\gal(L_{i}/K)\cong C_{3}$.

 Case 2: When $i>j\geq 1$ and
  $\disc(L_{i})\not\equiv\disc(L_{j}) \pmod{(K^{*})^{2}}$, the fields
  $L_{i}$ and $L_{j}$ are disjoint. Then there are three possibilities
  for the Galois group of their compositum, according to whether the
  discriminants are trivial (\textit{i.e.}, square).  In each case
  there are infinitely many primes which fulfill the condition, with
  density $\frac{4}{9}$ when $\gal(L_{i}L_{j})\cong S_{3}\times
  S_{3}$, and $\frac{5}{9}$ when $\gal(L_{i}L_{j})$ is $S_{3}\times
  C_{3}$.

Case 3: When $i,j\geq 1$ and
  $\text{disc}(L_{i})\equiv\text{disc}(L_{j})\pmod{(K^{*})^{2}}$ we
  have two possibilities; the density is $\frac{4}{9}$ when both
  Galois groups are isomorphic to $C_{3}$ and is $\frac{2}{9}$ when
  both are isomorphic to $S_{3}$.
\end{proof}

A distinguishing set $T_{0}$ of primes can be computed using the
following algorithm.  The size $t$ of~$T_{0}$ depends on the total
number $n$ of $C_{3}$ and $S_{3}$ extensions of $K$ unramified outside
$S$, and there exists such a set for which
$\lceil\log_{2}(n)\rceil\leq t\leq n-1$.

\begin{algorithm}[H]\label{algoritmoT0}
    \caption{To determine a distinguishing set $T_{0}$ of primes of
      $K$.}  \SetKwInOut{Input}{Input} \SetKwInOut{Output}{Output}
    \Input{A number field $K$. A finite set $S$ of primes of $K$. \\ A
      set $\FF=\{f_1,\dots,f_n\}$ of cubics defining $C_3$ and $S_3$
       \\\ extensions of $K$ unramified outside~$S$.}
    \Output{$T_0$, a distinguishing set of primes for $(\FF,S)$.}

	Let $f_{0}=x^3$\;
	Let $T_{0}=\{\}$\;
	\While{$\#\{\vv(f_{i},T_{0})\mid0\leq i\leq n\}<n+1$ }
	{
	Find $i\neq j$ such that $\vv(f_{i},T_{0})=\vv(f_{j},T_{0})$\;
	Find a prime $\p\notin S\cup T_{0}$ such that $\lambda(f_{i},\p)\neq \lambda(f_{j},\p)$\;
	Let $T_{0}= T_{0}\cup\{\p\}$\;
	}
    \Return $T_{0}$.
\end{algorithm}

\begin{example}\label{ex:1}
Let $K=\Q$ and take $S=\{2,37\}$.  The only $C_3$ extension of~$\Q$
unramified outside~$S$ is the splitting field of $f=x^{3} - x^{2} - 12
x - 11$ (with discriminant~$37^2$), while there are two such $S_3$
extensions with polynomials~$g=x^{3} - x^{2} - 3 x + 1$ and $h= x^{3}
- x^{2} - 12 x + 26$ (with discriminants $37\cdot2^2$ and
$-(2\cdot37)^2$ respectively), so $\FF=\{f,g,h\}$.  (These fields have
LMFDB labels \lmfdbnumberfield{3.3.148.1},
\lmfdbnumberfield{3.3.1369.1} and \lmfdbnumberfield{3.1.5476.1}.) We
  may take $T_0=\{3,5\}$ where the values of $\lambda$ are $(1, 1)$,
  $(1,0)$, $(0,1)$ for $f,g,h$ respectively.
\end{example}

\subsection{Determining residual irreducibility and splitting field}

As above, let $\rho$ be a Black Box $2$-adic Galois representation
over~$K$ unramified outside~$S$, let $\FF=\{f_1,\dots,f_n\}$ be a set
of irreducible cubics defining all~$C_3$ and~$S_3$ extensions of $K$
unramified outside~$S$, and let $T_0$ be a distinguishing set of
primes for~$(\FF,S)$.  For $1\le i\le n$ let $L_i$ be the splitting
field of~$f_i$ over~$K$, and let $L$ be the residual splitting field
of~$\rho$ with respect to one stable lattice.

\begin{proposition}\label{lemmac3s3}
With notation as above,
  \begin{enumerate}
  \item If $[L:K]=6$ or $3$ then, for exactly one value $i\ge1$, we
    have $L=L_{i}$ and
    $$
       \lambda(f_{i},\p)\equiv\tr(\rho(\frob{\p}))\pmod 2
    $$
    for all $\p\not\in S(\FF)$. Moreover, for infinitely many primes $\p$ we have
    $$
       \tr(\rho(\frob{\p}))\equiv1\pmod 2.
    $$
     \item $[L:K]\leq 2$ if and only if
       $$
       \tr(\rho(\frob{\p}))\equiv0\pmod 2
       $$
       for all $\p\notin S(\FF)$.
\end{enumerate}
\end{proposition}

\begin{proof}
Suppose that $[L:K]=6$ or $3$.  Then the image of $\overline{\rho}$ is
$C_{3}$ or $S_{3}$ and $L=L_{i}$, the splitting field of $f_{i}$, for
some $i$, $1\leq i\leq n$. Hence for all $\p\notin S(\FF)$,
by~Lemma~\ref{lemmafunctionlambda}, we have
\begin{align*}
 \lambda(f_{i},\p)=1&\Leftrightarrow \frobp\text{ has order } 3\text{ in Gal}(L_{i}/K)\\
 &\Leftrightarrow\rhobar(\frobp) \text{ has order }3\text{ in }\gl_{2}(\F_{2})\\
 &\Leftrightarrow\tr(\rho(\frobp))\equiv1\pmod 2.
\end{align*}

On the other hand, if $[L:K]\leq 2$, the image of $\rhobar$ is either
$C_{1}$ or $C_{2}$. Hence $\tr(\rho(\frob{\p}))\equiv 0\pmod2$ for all
$\p\not\in S$.
\end{proof}

Note that irreducibility of the residual representation can be
established with a single prime~$\p$ such that $\tr(\rho(\frob{\p}))$
is odd.  Using this proposition, we can achieve more: first, that for
$\rhobar$ to be reducible it suffices to check that
$\tr(\rho(\frob{\p}))$ is even for a \emph{finite} set of primes,
those in~$T_0$; secondly, that when they are not all even, the
values of $\tr(\rho(\frob{\p}))\pmod2$ for $\p\in T_0$ identify the
residual image precisely as $C_{3}$ or $S_{3}$, and also identify the
splitting field exactly.  Moreover both the set of cubics $\FF$ and
the distinguishing set~$T_0$ depend only on~$K$ and~$S$ and so may be
computed once and then used to test many representations~$\rho$ with
the same ramification restrictions.  The main result of this section
is as follows.

\begin{theorem}\label{thm:reducibleresidual}
Let $K$ be a number field, $S$ a finite set of primes of~$K$, and let
$\rho$ be a continuous $2$-dimensional $2$-adic Galois representation
over~$K$ unramified outside~$S$. Let $T_0$ be a distinguishing set
for~$S$ in the sense of Definition~\ref{defT0}.
\begin{enumerate}
\item
The finite set of values of $\tr(\rho(\frobp))\pmod{2}$, for~$\p\in
T_0$, determine the residual representation $\rhobar$ up to
semisimplification.  Hence (up to semisimplification) $\rhobar$ may be
identified from its Black Box presentation.
\item
In particular, the residual representation~$\rhobar$ has trivial
semisimplification (equivalently, is reducible over~$\F_2$), if and
only if
\[
   \tr(\rho(\frobp)) \equiv 0 \pmod 2 \qquad\forall\p\in T_0.
\]
\end{enumerate}
\end{theorem}

\begin{proof}
Let $\FF=\{f_1,\dots,f_n\}$ and let $T_0=\{\p_1,\dots,\p_t\}$ be a
distinguishing set for~$(S,\FF)$ as above.  The vectors
$$
  \vv_{i}=(\lambda(f_{i},\p_{1}),...,\lambda(f_{i},\p_{t})) \in \F_{2}^{t}
$$ for $1\le i\le n$ are distinct and non-zero by definition of~$T_0$.
  Using the Black Box, we compute the vector
$$
  \vv=(\tr(\rhobar(\frob {\p_{1}})),...,\tr(\rhobar(\frob {\p_{t}}))) \in \F_{2}^{t}.
  $$
By Proposition~\ref{lemmac3s3}, we have (with $L$ and $L_i$ as defined there)
$$\vv=\vv_{i} \Leftrightarrow L=L_{i}\Leftrightarrow [L:K]=6 \text{ or }3$$
and
$$
\vv=\0\Leftrightarrow [L:K]\leq 2.
$$ Hence $\rhobar$ is irreducible if and only if $\vv=\vv_i$ for
some~$i$, in which case its splitting field is that of~$f_i$ and its
image is isomorphic to~$S_3$, unless $\disc f_i\in (K^*)^2$ in which
case the image is~$C_3$.  Otherwise, $\vv=\0$ and $\rhobar$ is
reducible, with trivial semisimplification.
\end{proof}


\begin{algorithm}[H]\label{algoritmoc3/s3}
\caption{To determine the residual image of an integral $2$-adic
  Galois representation, up to semisimplification.}
\SetKwInOut{Input}{Input}
\SetKwInOut{Output}{Output}
\Input{A number field $K$.\\
       A finite set $S$ of primes of $K$.\\
       A Black Box Galois representation $\rho$ unramified outside $S$.
       }
\Output{$\bullet$ ({\tt True}, $f$, $G$) if $\rhobar$ is irreducible, with splitting
  field that of~$f$, and image $G\cong C_{3}$ or~$S_{3}$.\\
       $\bullet$ {\tt False} if $\rhobar$ is reducible.}

Let $\FF=\{f_{i}\}_{i=1}^{n}$ be a set of monic irreducible cubics
defining all $S_3$ and $C_3$ extensions of~$K$ unramified outside~$S$\;
Using Algorithm \ref{algoritmoT0}, compute a distinguishing set
$T_{0}=\{\p_1,\dots,\p_t\}$ of primes for $(\FF,S)$\;
Let $\vv=(\tr(\rhobar(\frob{\p_{1}})),...,\tr(\rhobar(\frob{\p_{t}})))$\;

\For{i=1...n}{
\If{$\vv=\vv(f_{i},T_{0})$}
{
  Let $G=C_3$ if $\disc f_i$ is square, else $G=S_3$\;
  \textbf{return} ({\tt True}, $f_{i}$, $G$).
}
}
\textbf{return} {\tt False}.
\end{algorithm}

\begin{example*}{\textbf{(continued.)}}
With $K=\Q$, $S=\{2,37\}$ we have $\#\FF=3$ and $T_0=\{3,5\}$.  Hence
for a mod~$2$ representation over~$\Q$ unramified outside~$\{2,37\}$
we may test irreducibility by inspecting the parity of the trace~$a_p$
at~$p=3$ and $p=5$.  As an example, we consider the $156$ isogeny
classes of elliptic curves of conductor $2^a37^b$.  Of these, $36$
have $a_3\equiv a_5\equiv0\pmod2$, hence the representation is
reducible; indeed, these curves have rational $2$-torsion.  There are
$8$ with $a_3\equiv a_5\equiv1\pmod2$ with 2-division field the
splitting field of~$f=x^{3} - x^{2} - 12 x - 11$.  The remaining~$112$
classes comprise $80$ with $a_3\equiv1, a_5\equiv0\pmod2$ and $32$
with $a_3\equiv0, a_5\equiv1\pmod2$, whose $2$-division fields have
Galois group~$S_3$ and are the splitting fields of~$g=x^{3} - x^{2} -
3 x + 1$ and $h= x^{3} - x^{2} - 12 x + 26$, respectively.  These
curves have no rational $2$-torsion.

Here we could have considered the curves up to isogeny and up to
quadratic twist, since quadratic twists obviously have the same
mod~$2$ representation.  The number of cases then reduces to~$22$ ($6$
reducible and $1$, $11$, and $4$ for each irreducible case).
\end{example*}
\section{Determining triviality of the residual representation up to isogeny}
\label{sec:small-large}

Let $\rho\colon G_{K}\to\gl_{2}(\Z_{2})$ be a continuous Galois
representation unramified outside $S$ with reducible residual
representation.  Depending on the choice of stable lattice~$\Lambda$,
the order of $\rhobar_\Lambda(G_K)\le\gl_2(\F_2)$ is either~$1$
or~$2$, though the semisimplification of~$\rhobar_\Lambda$ is always
trivial. In this section we will give a method to decide whether
within the isogeny class of~$\rho$ there is an integral
representation~$\rho_\Lambda$ whose residual
representation~$\rhobar_\Lambda$ is trivial.  If this is the case, it
follows from the remarks about the isogeny graph at the end of
Section~\ref{sec:background} that the corresponding vertex in the
isogeny graph~$\BT(\rho)$ has degree~$3$, the width of the graph is at
least~$2$, and it contains at least~$4$ vertices; otherwise, its width
is~$1$ and it consists of just two vertices linked by a single edge.
We call these \emph{large} and \emph{small} isogeny classes
respectively.

Vertices of~$\BT(\rho)$ either have degree~$1$, non-trivial residual
representation, and quadratic splitting field with non-trivial
discriminant in~$K(S,2)_u$; or degree~$3$ and trivial residual
representation.  So each vertex of~$\BT(\rho)$ has an associated
discriminant, and we would like to describe the graph structure
of~$\BT(\rho)$---the number of vertices, and width---as well as the
discriminants of its extremal (degree~$1$) vertices.

In this section we show how to distinguish the small and large cases;
in Section~\ref{sectionLargeIsogenyClasswidthat2} we will continue
under the assumption that the class is large.  The following notation
will be useful for the tests we will develop; note that since we are
now assuming that $\rho$ is residually reducible,
$\tr(\rho(\frobp))\equiv0\pmod2$ for all~$\p\notin S$ so that
$F_{\p}(1)\equiv 0\pmod2$.  Define
\begin{align}\label{ecuacionv(p)}
v(\p)=\ord_{2}(F_{\p}(1)).
\end{align}
When $v(\p)\geq k$ for some $k\ge1$, we define the \textit{test
  function}
\begin{align}\label{testequation}
t_{k}(\p)&=\dfrac{1}{2^{k}}F_{\p}(1)\pmod{2}\nonumber\\
&\phantom{:}=\dfrac{1}{2^{k}}(1-\tr(\rho(\frob{\p}))+\det(\rho(\frob{\p})))\pmod 2.
\end{align}
so that $t_{k}(\p)=0$ if and only if $v(\p)\geq k+1$. Write
$t_k(\sigma)=t_k(\p)$ when $\sigma=\frobp$.


\subsection{The test function for small isogeny classes}\label{elalgoritmo}
Let $\Lambda_{1}$ be a stable lattice under the action of
$\rho$. Since $\overline{\rho}$ is reducible, there is an index~$2$
sublattice $\Lambda_{2}$ which is also stable under $\rho$. Choosing
the bases $\Lambda_{1}=\left<v,w\right>$ and
$\Lambda_{2}=\left<v,2w\right>$ we have that
\[
\rho(\sigma) = \begin{pmatrix}a&b\\c&d \end{pmatrix}
               \equiv
               \begin{pmatrix}1&*\\ 0&1\end{pmatrix}\pmod2
\]
for all $\sigma\in G_{K}$. (Here we are showing matrices with respect
to the basis $\left<v,w\right>$, and our convention is that
$\begin{pmatrix}a&b\\c&d \end{pmatrix}$ maps $v\mapsto av+cw$ and
$w\mapsto bv+dw$.) There are two ways in which the graph
$\Lambda_{1}$---$\Lambda_{2}$ could be extended within~$\BT(\rho)$,
either or both of which could happen:
\begin{enumerate}
\renewcommand{\theenumi}{$(\alph{enumi})$}
\renewcommand{\labelenumi}{\textbf{\theenumi}}
\item If $c\equiv0\pmod4$ for all $\sigma\in G_{K}$ then \label{laconcha01}
\[
   \rho(\sigma) \equiv \begin{pmatrix}\pm1&*\\ 0& \pm1 \end{pmatrix}\pmod4
\]
and $\Lambda_{3}=\left<v,4w\right>$ is also stable, extending the
stable graph to $\Lambda_{1}$---$\Lambda_{2}$---$\Lambda_{3}$.  The
lattice $\Lambda_{4}=\left<2v,v+2w\right>$ is also stable and adjacent
to $\Lambda_{2}$, so $\Lambda_2$ has degree~$3$ in~$\BT(\rho)$.
\item If $b\equiv0\pmod2$ for all $\sigma\in G_{K}$ then \label{laconcha02}
\[
   \rho(\sigma) \equiv \begin{pmatrix} 1&0 \\ 0& 1 \end{pmatrix}\pmod2
\]
so $\overline{\rho}$ is trivial. Then $\Lambda_{3}'=\left<2v,w\right>$
is also stable and extends the graph to
$\Lambda_{3}'$---$\Lambda_{1}$---$\Lambda_{2}$.  The lattice
$\Lambda_{4}'=\left<v+w,2w\right>$ is also stable and adjacent to
$\Lambda_{1}$, so $\Lambda_1$ has degree~$3$ in~$\BT(\rho)$.
\end{enumerate}
These two situations are not essentially different, since by
conjugating with the matrix $\begin{pmatrix} 2&0\\ 0&1\end{pmatrix}$
  we interchange the roles of $\Lambda_{1}$ and $\Lambda_{2}$, and the
  two cases.

The following maps are easily seen to define two additive quadratic
characters of $G_{K}$, unramified outside $S$:
\begin{align*}
 \chi_{c}\colon\sigma\mapsto \dfrac{c}{2}\pmod2 \quad\text{and}\quad
 \chi_{b}\colon\sigma\mapsto b\pmod2,
\end{align*}
which correspond to two extensions $K(\sqrt{\Delta_{b}})$,
$K(\sqrt{\Delta_{c}})$ with $\Delta_{b}, \Delta_{c}\in K(S,2)_u$,
possibly equal or trivial, and the isogeny class~$\BT(\rho)$ is large
if and only if at least one is trivial.  This establishes the
following criterion.

\begin{proposition}\label{prop:small-large}
$\BT(\rho)$ is small if and only if the characters $\chi_{b}$ and
  $\chi_{c}$ are both non-trivial.
\end{proposition}

In order to turn this criterion into an algorithm we must see how to
obtain information about these two characters using only the Black Box
and a finite set of primes $\p\notin S$. Taking $k=1$ in
(\ref{testequation}) we use the test function
\begin{align}\label{test1}
t_{1}(\sigma)&= \frac{1}{2}(F_{\sigma}(1)) \pmod 2\nonumber\\
&\phantom{:}\equiv \frac{1}{2}(1-\tr(\rho(\sigma))+\det(\rho(\sigma))) \pmod 2.
\end{align}

\begin{proposition}\label{proposition5.1}
With notation as above,
\[
t_{1}(\sigma) = \chi_{b}(\sigma)\chi_{c}(\sigma).
\]
\end{proposition}
\begin{proof}
We compute
$t_{1}(\sigma)=\frac{1}{2}((a-1)(d-1)-bc)\equiv
bc/2\equiv\chi_{b}(\sigma)\chi_{c}(\sigma)\pmod{2},$
using $a\equiv d\equiv1\pmod{2}$.
\end{proof}

So the Black Box reveals the value of the \emph{product} of the two
additive characters.

\begin{corollary}
The following are equivalent, assuming that $\rho$ is residually reducible:
\begin{enumerate}
\item $\BT(\rho)$ is large;
\item $t_1(\sigma)=0$ for all~$\sigma\in G_K$;
\item $t_1(\p_I)=0$ for all primes~$\p_I$, one such prime for each of the $2^r$
  subsets $I\subseteq\{1,2,\dots,r\}$.
\end{enumerate}
\end{corollary}
\begin{proof}
The equivalence of the first two statements is because $\ker\chi_b$
and $\ker\chi_c$ are subgroups of~$G_K$, and no group is the union of
two proper subgroups. For the second equivalence, note that the pair
of values $(\chi_b(\frobp),\chi_c(\frobp))$ depends only on the
restriction of~$\frobp$ to the maximal elementary $2$-extension of~$K$
unramified outside~$S$ whose Galois group consists of
these~$\frob\p_I$.
\end{proof}

Although the corollary already reduces the current problem to a finite
number of tests, we will show in the next subsection how to use some
linear algebra over~$\F_2$ to reduce the test set of primes from a set
of size~$2^r$ (one for each subset~$I$) to a set of $r(r+1)/2$
\emph{quadratically independent} primes (with respect to~$S$).  Using
these, we will be able to determine not only whether at least one of
$\Delta_b$, $\Delta_c$ is trivial, in which case the class is large;
when both characters are non-trivial, we will also be able to determine
the unordered pair $\{\Delta_b,\Delta_c\}$ exactly.

\subsection{Quadratically independent sets of primes}\label{subsec:quadindep}

Let $\{\Delta_{i}\}_{i=1}^{r}$ be a basis for $V=K(S,2)_u$.  The
discriminants $\Delta_{b},\Delta_{c}\in V$ may be expressed as
$$
  \Delta_{b}=\prod_{i=1}^{r}\Delta_{i}^{x_{i}}, \quad
  \Delta_{c}=\prod_{i=1}^{r}\Delta_{i}^{y_{i}}
$$ with unknown exponent vectors $\x=(x_{i})$ and $\y=(y_{i})$ in
  $\F_{2}^{r}$.  We will determine the vectors $\x$ and $\y$ in the
  restricted sense of knowing whether either (a) at least one of $\x$
  and $\y$ is zero, or (b) they are both non-zero, in which case we
  will identify them precisely, as an unordered pair.

Let $T_{1}=\{\p_{1},...,\p_{r}\}$ be a linearly independent set of
primes chosen so that the $\alpha_{\p_{i}}$ are a dual basis to
$\{\Delta_{i}\}_{i=1}^{r}$. Then by $(\ref{alphaij=alphaialphaj})$ we
have $\chi_b(\p_i)=x_i$ and~$\chi_c(\p_i)=y_i$.  Hence, by Proposition
$\ref{proposition5.1}$, we have that $t_1(\p_i)=x_iy_i$.  More
generally for a prime~$\p_I$ (defined in Section~\ref{sec:characters})
we have, by (\ref{alphaij=alphaialphaj}),
\[
   t_1(\p_I) = x_Iy_I
\]
where we set $x_I=\sum_{i\in I}x_i$ and similarly for~$y_I$.

Define
\begin{align}\label{functionpsi}
\psi\colon V\times V\times V^{*}&\rightarrow\F_{2}\\\nonumber
(\Delta,\Delta',\alpha)&\mapsto \alpha(\Delta)\alpha(\Delta')
\end{align}
For fixed $\alpha$, the map $\psi_\alpha=\psi(-,-,\alpha)$ is a
symmetric bilinear function $V\times V\to\F_2$,
i.e., an element of the space $\text{Sym}^{2}(V)^{*}$ which has
dimension $r(r+1)/2$ and basis the functions $x_{i}y_{i}$ and
$x_{i}y_{j}+x_{j}y_{i}$ for $i\not=j$.  This leads us to define our
third (and last) set of test primes:

\begin{definition}\label{defnonquadset}
A set $T_{2}$ of primes $\p\not\in S$ is \emph{quadratically
  independent with respect to} $S$ if $\{\psi_{\alpha_{\p}}\mid \p\in
T_{2}\}$ is a basis for $\text{Sym}^{2}(V)^{*}$.
\end{definition}

The simplest quadratically independent sets consist of primes~$\p_i$
for $1\le i\le r$ (these already form a linearly independent set,
previously denoted~$T_1$), together with~$\p_{ij}$ for $1\le i<j\le
r$.  We will call quadratically independent sets of this form
\emph{special}.

\begin{remark}
If we fix instead $(\Delta,\Delta')$ in $(\ref{functionpsi})$ we
obtain a quadratic function
$\psi_{(\Delta,\Delta')}=\psi(\Delta,\Delta',-)$ on~$V^*$:
\begin{align*}
\psi_{(\Delta,\Delta')}\colon V^{*}&\rightarrow\F_{2}\\\nonumber
\alpha&\mapsto \alpha(\Delta)\alpha(\Delta').
\end{align*}
It is not hard to show that when $T_2$ is a quadratically independent set of
primes, the set $\{\alpha_{\p}\mid \p\in T_2\}$ is a
\emph{non-quadratic} subset of $V^{*}$ in the sense of Livn\'{e}
$\cite{livne}$.
\end{remark}

We now proceed to show that the values of the test function~$t_1(\p)$
for $\p$ in a special quadratically independent set of primes are
sufficient to solve our problem concerning the identification of the
vectors~$\x$ and~$\y$.  Define $\vv=(v_{1},...,v_{r})\in\F_{2}^{r}$ to
be the vector with entries
\[
 v_i=x_iy_i=t(\p_i).
 \]
Next let $\WW=(w_{ij})$ be the $r\times r$ matrix over $\F_{2}$ with
entries $w_{ii}=0$ and, for $i\not=j$,
\begin{align*}
  w_{ij} &= x_{i}y_{j}+x_{j}y_{i} = (x_i+x_j)(y_i+y_j) + x_iy_i +  x_jy_j\\
 &= t(\p_{ij})+t(\p_i)+t(\p_j).
\end{align*}
Then the $i$-th row of $\WW$ is given by
\begin{equation}
y_{i}\x+x_{i}\y \\ 
   = \begin{cases}
         \0 & \text{if } (x_{i},y_{i})=(0,0);\\
         \x & \text{if } (x_{i},y_{i})=(0,1);\\
         \y & \text{if } (x_{i},y_{i})=(1,0);\\
      \x+\y & \text{if } (x_{i},y_{i})=(1,1),\\
\end{cases}
\end{equation}
so that the rank of~$\WW$ is either~$0$ or~$2$.  Moreover,
\begin{itemize}
\item if $\x=\0$ or $\y=\0$, then $\vv=\0$ and $\WW=\0$;
\item if $\x\neq\0$ and $\y\neq\0$ and $\x=\y$, then $\vv=\x=\y\neq\0$ and $\WW=\0$;
\item if $\x\neq\0$ and $\y\neq\0$ and $\x\neq\y$, then
  $\WW\neq\0$. Moreover, at least two out of $\x$, $\y$, $\x+\y$
  (which are non-zero and distinct) appear as rows of $\WW$, and \label{cono03}
\begin{itemize}
\item if $\vv\neq \0$, then the rows of $\WW$ for which $v_{i}=1$ are
  $\x+\y$ and the remaining non-zero rows are equal to either $\x$ or $\y$;
\item if $\vv=\0$, then the non-zero rows of $\WW$ are all equal to either $\x$ and $\y$.
\end{itemize}
\end{itemize}
It follows that by inspecting $\vv$ and $\WW$, whose entries we can
obtain from our Black Box test function on $r(r+1)/2$ primes, we can
indeed determine whether $\x$ or $\y$ is zero, and if both are
non-zero then we can determine their values, and hence determine the
unordered pair of the discriminants~$\{\Delta_b, \Delta_c\}$.

\begin{proposition}\label{propo5.3}
Let $\rho$ be residually reducible.  From the set of values
$\{t_1(\p)\mid \p\in T_2\}$ of the test function~$t_1$ defined
in~(\ref{test1}), for $T_2$ a quadratically independent set of primes
with respect to~$S$, we may determine whether the isogeny class
of~$\rho$ is small or large, and in the first case we can determine
the unordered pair formed by the associated non-trivial discriminants.
\end{proposition}

See Algorithm~\ref{algoritmotest1}, where we follow the procedure
above, assuming that we take for $T_2$ a special set $\{\p_i\mid 1\le i\le
r\} \cup \{\p_{ij}\mid 1\le i<j\le r\}$.  In practice it might not be
efficient to insist on using a quadratically independent set of this
form, because we may need to test many primes $\p$ before finding
primes of the form $\{\p_{ij}\}$ for all $i<j$; also, the resulting
primes are likely to be large.  In applications, it may be
computationally expensive to compute the trace of $\rho(\frobp)$ for
primes $\p$ of large norm.  This is the case, for example, when $\rho$
is the Galois representation attached to a Bianchi modular form (see
\cite{pacetti} for numerical examples when $K$ is an imaginary
quadratic field of class number~$3$).  In our implementation we adjust
the procedure to allow for arbitrary quadratically independent sets.
The details are simply additional book-keeping, and we omit them here.

We give two algorithms to compute quadratically independent sets.  In
both cases we consider the primes of $K$ systematically in turn
(omitting those in~$S$), by iterating through primes on order of norm.
The first algorithm returns the
smallest such set (in terms of the norms of the primes), while the
second only uses primes for which $\#I(\p)\in\{1,2\}$ and returns a
set of the special form.

In Algorithm~\ref{algoritmoT2}, we construct a matrix $\Aa$ whose
columns are indexed by the subsets of $\{1,2,...,r\}$ of size $1$ and
$2$, i.e., the sets $\{i\}$ for $1\leq i\leq r$ and $\{i,j\}$ for
$1\leq i<j\leq r$, initially with $0$ rows. For each prime $\p$ we
compute $I(\p)$ and define $\vv(\p)$ in $\F_{2}^{\frac{r(r+1)}{2}}$ by
setting its coordinates to be
\begin{align}\label{vectorv(p)}
\begin{cases}
1& \text{ in position } i\text{ if } i\in I(\p)\\
1& \text{ in position } \{i,j\}\text{ if } \{i,j\}\subseteq I(\p)\\
0& \text{otherwise.}
\end{cases}
\end{align}
We add $\vv(\p)$ as a new row of $\Aa$, provided that this increases
the rank of $\Aa$, and we stop when $\rk\Aa=r(r+1)/2$.

\begin{algorithm}[H]\label{algoritmoT2}
\caption{To determine a quadratically independent set $T_{2}$ of primes of $K$.}
\SetKwInOut{Input}{Input}
\SetKwInOut{Output}{Output}
\Input{A number field $K$.\\
       A finite set $S$ of primes of $K$.}
\Output{A finite quadratically independent set $T_{2}$ of primes of $K$.}
Let $\{\Delta_{i}\}_{i=1}^{r}$ be a basis for $K(S,2)_u$\;
Let $T_{2}=\{\}$\;
Let $\Aa$ be a $0\times \frac{r(r+1)}{2}$ matrix over $\F_{2}$\;
\While{$\Aa$ \emph{has} $< r(r+1)/2$ \emph{rows}}{
Let $\p$ be a prime not in $S\cup T_{2}$\;
Compute $I=I(\p)$ using $(\ref{definitionI(p)})$\;
Compute $\vv(\p)$ from $(\ref{vectorv(p)})$\;
Let $\Aa'=\Aa+\vv(\p)$   (adjoin $\vv(\p)$ as a new row of $\Aa$)\;
\If{$\rk(\Aa')>\rk(\Aa)$}
{Let $\Aa=\Aa'$\;
Let $T_{2}=T_{2}\cup\{\p\}$.}
}
\textbf{return} $T_{2}$.
\end{algorithm}

This variant produces a special quadratically independent set by only
including primes~$\p$ for which $I(\p)$ has size~$1$ or~$2$.

\begin{algorithm}[H]\label{algoritmoT2*}
\caption{To determine a special quadratically independent set $T_{2}$ of primes of $K$.}
\SetKwInOut{Input}{Input}
\SetKwInOut{Output}{Output}
\Input{A number field $K$.\\
       A finite set $S$ of primes of $K$.}
\Output{An indexed special quadratically independent set $T_{2}$ of primes. }
Let $A=B=\{\}$\;
Let $T_{2}=\{\}$\;
\While{$\#(A\cup B)<r(r+1)/2$}
{
Let $\p$ be a prime not in $S\cup T_{2}$\;
Compute $I=I(\p)$ using $(\ref{definitionI(p)})$\;
\If{$\# I=1$ \emph{with} $I=\{i\}$}
 {\If{$i\notin A$}
   {
   Let $\p_{i}=\p$\;
   Let $A=A\cup\{i\}$\;
   Let $T_{2}=T_{2}\cup\{\p_{i}\}.$
   }
 }
\If{$\# I=2$ \emph{with} $I=\{i,j\}$ \emph{and} $i<j$}
 {\If{$(i,j)\notin B$}
   {
   Let $\p_{ij}=\p$\;
   Let $B=B\cup\{(i,j)\}$\;
   Let $T_{2}=T_{2}\cup\{\p_{ij}\}.$
   }
 }
}
\textbf{return} $T_{2}$.
\end{algorithm}

\begin{example*}{(continuation of Example~\ref{ex:1})}
As before, we take $K=\Q$ and $S=\{2,37\}$.  Using $[-1,2,37]$ as an
ordered basis for $K(S,2)=K(S,2)_u$ we find, using
Algorithm~\ref{algoritmoT2*}, $T_2=\{7,53,17,3,5,23\}$.  For example,
$p=23$ is inert in $\Q(\sqrt{d})$ for $d=-1$ and $d=37$ but not for
$d=2$, so $I(23)=\{1,3\}$. The data for these primes is as follows:
\[
\begin{tabular}{c|cccccc}
$I$ & $\{1\}$& $\{2\}$& $\{3\}$& $\{1,2\}$& $\{2,3\}$& $\{1,3\}$\\
\hline
  $p_I$ & $7$ & $53$ & $17$ & $3$ & $5$ & $23$ \\
\end{tabular}
\]
Applying Algorithm~\ref{algoritmotest1} to the 36 isogeny classes of
elliptic curves with good reduction outside~$\{2,37\}$ and rational
$2$-torsion, we find that in $4$ cases the class is large, so contains
an elliptic curve with full $2$-torsion defined over~$\Q$ and hence
trivial mod-$2$ representation (these classes have LMFDB labels
\lmfdbeciso{32}{a}, \lmfdbeciso{64}{a}, \lmfdbeciso{43808}{a}, \lmfdbeciso{87616}{z}); while in all other
cases the class is small.  The discriminant pairs
$\{\Delta_1,\Delta_2\}$ returned by Algorithm~\ref{algoritmotest1} in these cases are
$\{-1,37\}$ (4 cases); $\{37,-37\}$ (8 cases); $\{-1,2\}$ (8 cases);
$\{2,-2\}$ (4 cases); $\{2,2\}$ (4 cases); and $\{74,-74\}$ (4 cases).
For example, isogeny class \lmfdbeciso{350464}{h} gives $a_p=0$ for $p=5, 7,
23$ and~$53$ while $a_{17}=6$ and $a_3=2$; this yields $\vv=(0,1,0)$
and $\WW=\0$, so both discriminants are~$2$ (modulo squares).  Indeed,
this isogeny class consists of two elliptic curves linked by
$2$-isogeny, each having a discriminant which is twice a square.

We leave it to the reader to explain why in every case the Hilbert
Symbol $(\Delta_1,\Delta_2)=+1$.
\end{example*}

\begin{algorithm}[H]\label{algoritmotest1} 
\caption{To determine whether the stable Bruhat-Tits tree of $\rho$
  has width exactly $1$ or at least $2$, together with the associated
  discriminants.}
\SetKwInOut{Input}{Input}
\SetKwInOut{Output}{Output}
\Input{A number field $K$.\\
       A finite set $S$ of primes of $K$.\\
       A Black Box Galois representation $\rho$ unramified outside $S$\\ 
       whose residual image is reducible.}
\Output{If $BT(\rho)$ has width $1$, return: {\tt True}, $\{\Delta_1, \Delta_2\}$.\\
       If $BT(\rho)$ has width $\geq 2$, return: {\tt False}.}
Let $\{\Delta_1,\dots,\Delta_r\}$ be a basis for $K(S,2)_u$\;
Let $T_{2}=\{\p_i\mid 1\le i \le r\}\cup\{\p_{ij}\mid 1\le i<j\le r\}$
 be a special quadratically independent set for $S$\;
Let $\vv=(t_{1}(\p_{1}),...,t_{1}(\p_{r}))\in \F_{2}^{r}$\;
Let $\WW=(t_{1}(\p_{ij})+t_{1}(\p_{i})+t_{1}(\p_{j}))\in M_{r}(\F_{2})$\;
\If{$\WW=\0$ and $\vv=\0$}{
\textbf{return} {\tt False}\;
}
\eIf{$\WW=\0$}{
Let $\x=\y=\vv$\;
}
{
\eIf{$\vv=\0$}{
Let $\x$ and $\y$ be any two distinct non-zero rows of $\WW$.
}
{
Let $\z$ be the $i$th row of $\WW$, where $i$ is such that $t_1(\p_i)=1$\;
Let $\x$ be any non-zero row of $\WW$ distinct from $\z$\;
Let $\y=\x+\z$.
}
}
\textbf{return} True, $\{\prod_{i=1}^{r}\Delta_{i}^{x_{i}}, \prod_{i=1}^{r}\Delta_{i}^{y_{i}} \}$.
\end{algorithm}
~ 
The methods of this section give an algorithm to determine whether the
isogeny class of~$\rho$ contains an integral representation whose
residual representation is trivial.
\begin{theorem}\label{thm:trivialmod2}
Let $K$ be a number field, $S$ a finite set of primes of~$K$, and let
$\rho$ be a continuous $2$-dimensional $2$-adic Galois representation
over~$K$ unramified outside~$S$. Assume that $\rho$ has reducible
residual representation.  Then there exists a stable lattice with
respect to which the residual representation $\rhobar$ is trivial, if
and only if
\[
    t_1(\p)\equiv0\pmod2 \qquad \forall \p \in T_2;
\]
that is,
\[
   1 - \tr\rho(\frobp) + \det\rho(\frobp) \equiv0\pmod4 \qquad \forall \p \in T_2;
\]
where $T_2$ is any quadratically independent set of primes for~$S$.
\end{theorem}

\goodbreak

\section{Large isogeny classes}\label{sectionLargeIsogenyClasswidthat2}
From now on we will assume that $\rho$ has trivial residual
representation, so that its isogeny class $\BT(\rho)$ consists at
least of $\rho$ together with the three $2$-isogenous integral
representations: recall that each lattice~$\Lambda$ has three
sublattices, and the condition that $\rhobar_\Lambda$ is trivial is
equivalent to each of these being stable.  The next step is to
determine whether the class is larger than this, i.e., whether it has
width greater than~$2$.  This is not the case if and only if each of
the $2$-isogenous representations has a non-trivial discriminant (as
defined in the previous section), in which case we would like to
determine this (unordered) set of three discriminants.  Furthermore,
we would like to determine $\rho\pmod4$ completely.

It turns out that it is no more work to deal with the more general
situation, where we assume that $\rho\pmod{2^k}$ is trivial for some
$k\ge1$, and determine $\rho\pmod{2^{k+1}}$ completely.  The
description of $\rho\pmod{2^{k+1}}$ will be in terms of a collection
of four additive quadratic characters, which we will be able to
determine using only the values of $F_\p(1)$ for $\p$ in the same
quadratically independent set~$T_2$ used in the previous section.  The
reason for this is that $\gl(\Z/2^{k+1}\Z)$ is an extension of
$\gl(\Z/2^{k}\Z)$ by~$M_2(\F_2)$, which is (as additive group) an
elementary abelian of order~$2^4$, as can be seen by the following
short exact sequence:
\[
    0 \longrightarrow M_2(\F_2) \longrightarrow \gl(\Z/2^{k+1}\Z)
    \longrightarrow \gl(\Z/2^{k}\Z) \longrightarrow 1
\]
where the second arrow maps $A\in M_2(\F_2)$ to $\II+2^kA\in
\gl(\Z/2^{k+1}\Z)$.

\bigskip

Thus let $\rho\colon G_{K}\rightarrow\gl_{2}(\Z_{2})$ be an integral
Galois representation unramified outside $S$, and assume that $\rho$
is trivial modulo $2^{k}$ for some positive integer $k$.  Write
\begin{align}\label{rhomod2k}
\rho(\sigma)=\II+2^{k}\mu(\sigma),
\end{align}
where
$$
 \mu(\sigma) = \begin{pmatrix} a(\sigma) & b(\sigma)\\
                               c(\sigma)&d(\sigma) \end{pmatrix}
 \in M_{2}(\Z_{2}).
$$ Then $F_\sigma(1)=2^{2k}\det\mu(\sigma)\equiv0\pmod{2^{2k}}$, and
 we can use the test function $t_{2k}(\p)=\frac{1}{2^{2k}}F_\p(1) =
 \det\mu(\sigma) \equiv ad-bc \pmod2$ for $\p\notin S$.

Secondly, with the same notation,
\[
   \det\rho(\sigma) \equiv 1+2^k(a+d) \pmod{2^{2k}},
\]
so
\[
   a+d \equiv \frac{1}{2^k}(\det\rho(\sigma) -1) \pmod{2}.
\]
Thus we see that the Black Box gives us the values of
both~$\tr\mu(\sigma)$ and $\det\mu(\sigma)\pmod2$ for
$\sigma=\frobp\in G_K$.  Now the map $\sigma \mapsto
\mu(\sigma)\pmod2$ is a group homomorphism $G_{K}\to M_{2}(\F_{2})$;
composing with the four characters
\begin{align*}
M_{2}(\F_{2})&\rightarrow\F_{2}\\
\begin{pmatrix}a&b\\c&d \end{pmatrix}&\mapsto a,b,c,d
\end{align*}
we obtain four additive characters of~$G_K$
\begin{align*}
G_{K}&\rightarrow\F_{2}\\
\sigma &\mapsto a(\sigma),b(\sigma),c(\sigma),d(\sigma)\pmod 2
\end{align*}
all unramified outside $S$, which we denote by
$\chi_{a},\chi_{b},\chi_{c}$ and $\chi_{d}$. To each character there
is associated a discriminant, named
$\Delta_{a},\Delta_{b},\Delta_{c},\Delta_{d}\in K(S,2)_u$.  Set
$\chi_{abcd}=\chi_{a}+\chi_{b}+\chi_{c}+\chi_{d}$ and
$\chi_{\det}=\chi_{a}+\chi_{d}$; the latter has discriminant
$\Delta_{\det}=\Delta_a\Delta_d$ (the reason for this notation will be
clear after the following lemma).  Our task is to use the values of
$a+d$ and $ad-bc$ at suitably chosen primes to obtain information
about these four characters.

The previous computation of determinants gives the following result
linking $\tr\mu(\sigma)=a+d$ with $\det\rho(\sigma)\pmod{2^{k+1}}$.
Recall that by equality of discriminants we always mean modulo
squares.
\begin{lemma}
Assume that $\rho$ is trivial modulo~$2^k$.  With notation as above,
the following are equivalent:
\begin{enumerate}
\item $\det\rho$ is trivial modulo~${2^{k+1}}$;
\item $a(\sigma) \equiv d(\sigma)\pmod2$ for all~$\sigma\in G_K$;
\item $\chi_{\det}=0$;
\item $\Delta_{\det}=1$.
\end{enumerate}
\end{lemma}

The characters we have just defined depend not only on the stable
lattice (here $\Lambda=\Z_2^2$, since we are treating $\rho$ as an
integral matrix representation) but also on a choice of basis.  If we
change basis via~$\UU\in\gl_2(\Z_2)$, the result is to
conjugate the matrices $\rho(\sigma)$ and $\mu(\sigma)$ by~$\UU$ and
replace the four characters $\chi_a$, \dots, $\chi_d$ by $\F_2$-linear
combinations.  By using suitable matrices~$\UU$ of orders~$2$ and~$3$
we may obtain all~$6$ permutations of $\{b, c, a+b+c+d\}$: taking
$\UU=\begin{pmatrix}-1&-1\\1&0 \end{pmatrix}$ (of order~$3$) cycles
$b\mapsto c\mapsto a+b+c+d\mapsto b$, while
$\UU=\begin{pmatrix}0&1\\1&0 \end{pmatrix}$ (of order~$2$)
transposes~$b\leftrightarrow c$ while fixing~$a+b+c+d$.  Of course the
determinant character~$a+d$ (which is the sum of these three) is
unchanged.  We will make use of this symmetry in what follows.

More generally, if $\UU\in\gl_2(\Q_2)\cap M_2(\Z_2)$ is such that
conjugation by~$\UU$ maps the image of~$\rho$ into $\gl_2(\Z_2)$, then
$\sigma\mapsto\UU\rho(\sigma)\UU^{-1}$ is another integral
representation isogenous to~$\rho$.  We will use this construction
below with $\UU=\begin{pmatrix}2&0\\0&1 \end{pmatrix}$.

\subsection{Stable sublattices of index~$2^{k+1}$}
We continue to assume that $\rho$ is trivial modulo~$2^k$ and use the
notation introduced in the previous subsection.  Clearly all
sublattices of index~$2^k$ in $\Lambda=\Z_2^2$ are stable
under~$\rho$.  Here we consider the sublattices of index $2^{k+1}$ and
show that the condition of whether they are also stable may be
expressed in terms of the characters~$\{\chi_b, \chi_c,
\chi_{a+b+c+d}\}$.  In terms of the isogeny graph~$\BT(\rho)$, it
contains all paths of length~$k$ (of which there are $3\cdot2^{k-1}$)
starting at the ``central'' vertex associated with~$\Lambda$---so the
graph has width at least~$2k$---and we are determining whether any
such paths may be extended within~$\BT(\rho)$ by one edge.  This turns
out to depend only on the first edge in the path (adjacent
to~$\Lambda$ itself).

When considering sublattices we restrict to those which are
\emph{cocyclic}, i.e. for which the quotient is cyclic, or
equivalently are not contained in~$2\Lambda$.  The cocyclic
sublattices $\Lambda'$ of index $2^{k+1}$ in $\Lambda=\Z_{2}^{2}$ are
given by
\begin{align*}
\Lambda'=\langle \vv\rangle+2^{k+1}\Lambda,\quad\text{with}\quad
\vv=\begin{pmatrix}x\\ y\end{pmatrix}\in\Z_{2}^{2},
\end{align*}
where $x,y$ are not both even, and $\Lambda'$ only depends on the
image of~$\vv$ in $\P^1(\Z/2^{k+1}\Z)$.  Now $\Lambda'$ is fixed by
$\rho$ if and only if for all $\sigma\in G_{K}$
\[
  \rho(\sigma)\vv\equiv \lambda \vv\pmod{2^{k+1}}
\]
for some $\lambda\in\{1,1+2^{k}\}$.  Since
$\rho(\sigma)=\II+2^k\mu(\sigma)$, this is if and only if~$\vv$ is an
eigenvector of~$\mu(\sigma)\pmod2$.  Hence the stability of~$\Lambda'$
only depends on the image of~$\vv$ in~ $\P^1(\Z/2\Z)$, and the three
possible values of~$\vv\pmod2$ correspond to the three edges in the
graph adjacent to~$\Lambda$ itself.  The following is now immediate
(where to save space we write~$\vv$ as a row vector):
\begin{lemma}
\begin{enumerate}
\item $\vv\equiv (1,0)\pmod2$ is an eigenvector of~$\mu(\sigma)$ if
  and only if $c(\sigma)\equiv0\pmod2$; hence such~$\Lambda'$ are
  stable if and only if $\chi_c=0$;
\item $\vv\equiv (0,1)\pmod2$ is an eigenvector of~$\mu(\sigma)$ if
  and only if $b(\sigma)\equiv0\pmod2$; hence such~$\Lambda'$ are
  stable if and only if $\chi_b=0$;
\item $\vv\equiv (1,1)\pmod2$ is an eigenvector of~$\mu(\sigma)$ if
  and only if $a(\sigma)+b(\sigma)+c(\sigma)+d(\sigma)\equiv0\pmod2$;
  such~$\Lambda'$ are stable if and only if $\chi_{a+b+c+d}=0$.
\end{enumerate}
\end{lemma}

For example, when $k=1$, the generic stable Bruhat-Tits tree of width at least $2$ looks like
\begin{figure}[H]
\centering
\begin{tikzpicture}[minimum width=1cm]
\node(0) at (1,5.5) {\phantom{0}};
\node(1) at (3,5.5) {\phantom{1}};
\node[circle,draw](b) at (2,4.5) {${}_{\phantom{cd}} \Delta_{b\phantom{d}}$};
\node[circle,draw](det) at (2,2)  {$~~~1~~~$};
\node[circle,draw](c) at (0,0)  {${}_{\phantom{ab}}\Delta_{c\phantom{d}}$};
\node(2) at (1,-1) {\phantom{2}};
\node(3) at (-1,-1) {\phantom{3}};
\node[circle,draw](abcd) at (4,0) {$\Delta_{abcd}$};
\node(4) at (3,-1) {\phantom{4}};
\node(5) at (5,-1) {\phantom{5}};

\draw [draw,-] (abcd) -- (det);
\draw [draw,-] (det) -- (c);
\draw [draw,-] (det) -- (b);
\draw [thick,dotted,-] (b) -- (0);
\draw [thick,dotted,-] (b) -- (1);
\draw [thick,dotted,-] (c) -- (2);
\draw [thick,dotted,-] (c) -- (3);
\draw [thick,dotted,-] (abcd) -- (4);
\draw [thick,dotted,-] (abcd) -- (5);
\end{tikzpicture}
\makeatletter 
\renewcommand{\thefigure}{A\@arabic\c@figure}
\makeatother
\caption{Tree of width at least $2$.}\label{width2tree}
\end{figure}
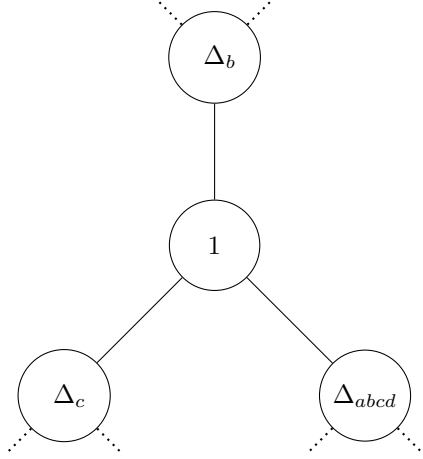
\noindent Here, each vertex has been labelled with its discriminant
in~$K(S,2)_u$, as defined in the previous section.  Note that the
three discriminants at the vertices adjacent to the central one (which
has trivial discriminant) have product~$\Delta_{\det}$, only depending
on~$\det\rho$.

In the case $k=1$ we deduce the following.
\begin{corollary}
When $\rho$ is trivial modulo~$2$, the isogeny graph~$\BT(\rho)$ has
width at least~$3$ if and only if at least one of the characters
$\chi_b$, $\chi_c$, $\chi_{a+b+c+d}$ is trivial.
\end{corollary}
Below we will see how to determine all four characters (up to $S_3$
symmetry).  In the case $k=1$, we will determine when all three
characters in the Corollary are non-trivial, so that the graph has
width exactly~$2$, and in this case we will determine precisely the
unordered set of three discriminants in the diagram.

\subsection{Determining the four characters: the test}\label{discriminants}
As before, let $\{\Delta_{i}\}_{i=1}^{r}$ be a fixed basis of $K(S,2)_u$ and write
\begin{align}
\Delta_{b}=\prod_{i=1}^{r}\Delta_{i}^{x_{i}}, \qquad \Delta_{c}=\prod_{i=1}^{r}\Delta_{i}^{y_{i}}, \qquad \Delta_{abcd}=\prod_{i=1}^{r}\Delta_{i}^{z_{i}},
\end{align}
$$\qquad \Delta_{a}=\prod_{i=1}^{r}\Delta_{i}^{u_{i}}, \qquad \Delta_{d}=\prod_{i=1}^{r}\Delta_{i}^{v_{i}},$$
where
$$\x=\{x_{i}\}_{i=1}^{r}, \y=\{y_{i}\}_{i=1}^{r},
\z=\{z_{i}\}_{i=1}^{r},\uu=\{u_{i}\}_{i=1}^{r},\vv=\{v_{i}\}_{i=1}^{r}\in\F_{2}^{r}.$$
To determine $\rho$ modulo $2^{k+1}$ it is enough to determine these
vectors, noting that $\x+\y+\z=\uu+\vv$, and bearing in mind the $S_3$
symmetry.

For primes $\p\notin S$, we use the test function $t_{2k}(\p)\equiv
ad-bc\pmod2$, dividing into cases according to
$\det\rho(\frobp)\pmod{2^{k+1}}$:
\begin{itemize}
\item If $\det(\rho(\frob{\p}))\equiv 1\pmod{2^{k+1}}$, then
  $a+d\equiv 0\pmod 2$, hence $ad\equiv a\pmod2$, so
\begin{align}\label{t2k1det1}
t_{2k}(\p)\equiv a+bc\pmod 2.
\end{align}
\item If $\det(\rho(\frob{\p}))\equiv 1+2^{k}\pmod{2^{k+1}}$, then
  $a+d\equiv 1\pmod 2$, so $ad\equiv0\pmod2$, and
\begin{align}\label{t2k1det12k}
t_{2k}(\p)\equiv bc\pmod 2.
\end{align}
\end{itemize}
Note that we will know from the Black Box which case we are in from
the value of~$\det\rho(\frobp)$.  We also note for later reference
that from
\begin{align}\label{traceofrho}
\tr(\rho(\frob{\p}))=2+2^{k}(a+d)
\end{align}
we can obtain the exact value of $a+d$: later we will need
$a+d\pmod4$.

Now it is convenient to divide into two cases, depending on whether or
not~$\det\rho$ is trivial modulo~$2^{k+1}$; equivalently, whether or
not~$\Delta_{\det}=1$.

\subsection{Determining the four characters: the case $\Delta_{\det}=1$}\label{section1mod2k1}

In this case the character $\chi_{\det}$ is trivial,
$\Delta_{a}=\Delta_{d}$, and $\uu=\vv$. Moreover,
$\Delta_{abcd}=\Delta_{b}\Delta_{c}$, so $\x+\y+\z=\0$.  By $S_3$
symmetry, only the set~$\{\x,\y,\z\}$ is well-defined.

Taking~$T_2=\{\p_i\mid 1\le i\le r\} \cup \{\p_{ij}\mid 1\le i<j\le
r\}$ as in Section~\ref{sec:small-large}, we have
\begin{align}\label{testfunctionst2kmod1k}
t_{2k}(\p_{i})&=u_{i}+x_{i}y_{i},\quad i\geq 1,\\\nonumber
t_{2k}(\p_{ij})&=u_{i}+u_{j}+(x_{i}+x_{j})(y_{i}+y_{j}),\quad i,j\geq 1.
\end{align}
Define
\begin{align}
w_{ij}&=x_{i}y_{j}+x_{j}y_{i}\\\nonumber
&=t_{2k}(\p_{i})+t_{2k}(\p_{j})+t_{2k}(\p_{ij}),\quad i,j\geq 1
\end{align}
and construct the matrix $\WW=(w_{ij})\in M_r(\F_2)$.  Each non-zero
row of~$\WW$ is equal to one of~$\x$, $\y$ or~$\z$, and as in
Section~\ref{sec:small-large}, if $\WW\not=0$ then $\WW$ has at least
two distinct non-zero rows and has rank~$2$.

\smallskip
\noindent\textbf{Case 1.} $\rk\WW=2$.  Now $\WW$ contains at least two
distinct non-zero rows, which by symmetry we can take to be the values
of~$\x$ and~$\y$.  Then $\z=\x+\y$, and we obtain the value of $\uu$
(which equals~$\vv$), using (\ref{testfunctionst2kmod1k}) and the now
known values of $\x$ and $\y$. Therefore we have computed all the
exponent vectors $\uu,\vv,\x,\y,\z$ and obtained
$\Delta_{a},\Delta_{b},\Delta_{c},\Delta_{d}$ and $\Delta_{abcd}$.

\smallskip
\noindent\textbf{Case 2.} $\WW=\0$. Now at least one of~$\x$, $\y$
or~$\z$ is zero; by symmetry we may take $\y=\0$, and $\x=\z$, but we
do not yet know the common value of $\x$ and $\z$.  However we have
$t_{2k}(\p_{i})=u_{i}+x_{i}y_{i}=u_i$, so we recover~$\uu$.

To determine~$\x$ and hence obtain the final discriminant
$\Delta_{b}$, we need to go a step further and consider the values of
$F_\p(1)\pmod{2^{2k+2}}$.  At the end we may need to replace~$\rho$ by
a $2$-isogenous representation; recall that the Black Box only
determines~$\rho$ up to isogeny, so this is valid.

Recalling the notation of $(\ref{rhomod2k})$, since $\y=\0$ we
observe that the entry $c$ is always even; put $c=2c_{1}$. Denote
by~$\chi_{c_1}$ the character $\sigma\mapsto c_1(\sigma)\pmod2$ and
let~$\Delta_{c_1}$ be its discriminant.  From the information already
known and further tests using the Black Box with the same primes
in~$T_2$ but to higher $2$-adic precision, we can determine the values
of the product $\chi_b\chi_{c_1}$.  As in
Section~\ref{sec:small-large}, we can then determine whether
either~$\Delta_b$ or~$\Delta_{c_1}$ is trivial, and their values if
both are non-trivial.  In the first case we may assume (conjugating if
necessary) that $\Delta_b=1$ (equivalently, $\x=\0$).  In the second
case, we may take either of the non-trivial discriminants to
be~$\Delta_b$.  This apparent ambiguity is illusory, since we are free
to replace the initial integral representation~$\rho$ by an isogenous
one.

For $\p\notin S$ we have
\begin{align}\label{senka}
F_{\p}(1)=2^{2k}(ad-2bc_{1}).
\end{align}
In order to proceed, we will need the value of $ad\pmod4$.  Recall
that we know the exact value of $a+d$ from~$(\ref{traceofrho})$, and
we also know the common parity of~$a$ and $d$, namely $u_I$ if
$\p=\p_I$.

\begin{enumerate}
\item If $\p$ is such that $a\equiv d\equiv 0\pmod 2$, then
  $ad\equiv0\pmod4$ and we obtain
\begin{align*}
F_{\p}(1)&\equiv2^{2k+1}bc_{1}\pmod{2^{2k+2}},
\end{align*}
so our standard test function
\begin{align}\label{senka2.1}
t_{2k+1}(\p)=\frac{F_{\p}(1)}{2^{2k+1}}\equiv bc_{1}\pmod{2}
\end{align}
gives the required value.
\item If $\p$ is such that $a\equiv d\equiv 1\pmod2$ and $a+d\equiv
  0\pmod 4$, then $ad\equiv -1\pmod 4$, so $(\ref{senka})$ becomes
\begin{align*}
F_{\p}(1) &\equiv -2^{2k}+2^{2k+1}bc_{1} \pmod{2^{2k+2}}.
\end{align*}
Hence we define a modified test function as follows:
\begin{align}\label{senka2.2}
\tilde t_{2k+1}(\p)=\dfrac{F_{\p}(1)+2^{2k}}{2^{2k+1}}\equiv bc_{1}\pmod{2}.
\end{align}
\item If $\p$ is such that $a\equiv d\equiv 1\pmod2$ and $a+d\equiv
  2\pmod 4$, then $ad\equiv 1\pmod 4$ and $(\ref{senka})$ becomes
\begin{align*}
F_{\p}(1)&\equiv 2^{2k}+2^{2k+1}bc_{1} \pmod{2^{2k+2}};
\end{align*}
we define
\begin{align}\label{senka2.3}
\tilde t_{2k+1}(\p)=\dfrac{F_{\p}(1)-2^{2k}}{2^{2k+1}}\equiv bc_{1}\pmod{2}.
\end{align}
\end{enumerate}

In summary, when $\rho$ is trivial modulo~$2^{k}$ and has trivial
determinant modulo~$2^{k+1}$, we can use the test function values
$t_{2k}(\p)$ for $\p\in T_2$ (where $T_2$ is a quadratically
independent set of primes for~$S$), together with either $t_{2k+1}$ or
one of the modified tests $\tilde t_{2k+1}$ depending on~$\p$, to
determine the full set of characters $\chi_a$, $\chi_b$, $\chi_c$,
$\chi_d$, satisfying $\chi_a+\chi_d=0$, if necessary replacing $\rho$
by a $\gl_2(\Z_2)$-equivalent representation, or by a $2$-isogenous
representation.  In particular, if all the characters are trivial then
(up to a $2$-isogeny) we conclude that $\rho$ is trivial
modulo~$2^{k+1}$.

\subsection{Determining the four characters: the case $\Delta_{\det}\not=1$}\label{section12kmod2k1}

Now assume that the determinant character $\chi_{\det}$ is
non-trivial, i.e. that $\det\rho$ is not
identically~$1\pmod{2^{k+1}}$.  To ease notation, we choose a basis
$\{\Delta_{i}\}_{i=1}^{r}$ of $K(S,2)_u$ such that
$\Delta_{1}=\Delta_{\det}$.  The unknown vectors in~$\F_2^r$ then satisfy
\begin{align*}
\x+\y+\z=\uu+\vv=\ee_{1},
\end{align*}
where $\ee_{1}=(1,0,...,0)$.  Denote by $\x'$, $\y'$ \textit{etc.} the
vectors in~$\F_2^{r-1}$ obtained by deleting the first coordinate.
These satisfy
\begin{align*}
\x'+\y'+\z'=\uu'+\vv'=\0
\end{align*}
and we will determine them first.

Take primes $\p_{i},\p_{ij}\in T_{2}$ with $i, j\geq 2$ and~$i\neq j$.
For such primes (as for all $\p_I$ when $1\notin I$) we have
$\det\rho(\frobp)\equiv1\pmod{2^{k+1}}$, so from (\ref{t2k1det1}) and
using $u_i=v_i$ for $i\ge2$ we see that
\begin{align}\label{testfunctionst2kmod2k}
t_{2k}(\p_{i})&=u_{i}+x_{i}y_{i}, \quad i\geq 2,\\\nonumber
t_{2k}(\p_{ij})&=u_{i}+u_{j}+(x_{i}+x_{j})(y_{i}+y_{j}), \quad 2\leq i\neq j\leq r,
\end{align}
and hence we can compute
\begin{align}
w_{ij}&=x_{i}y_{j}+x_{j}y_{i}\\\nonumber
&=t_{2k}(\p_{i})+t_{2k}(\p_{j})+t_{2k}(\p_{ij}),\quad i,j\geq 2.
\end{align}
Just as in Section~\ref{section1mod2k1} we can determine the shortened
vectors~$\x',\y',\z',\uu',\vv'$ (possibly replacing~$\rho$ by an
isogenous representation).

The final step is to determine the first coordinates $u_{1}$, $v_{1}$,
$x_{1}$, $y_{1}$ and $z_{1}$ with $x_{1}+y_{1}+z_{1}=u_{1}+v_1=1$,
using the remaining primes in~$T_2$ and test values $t_{2k}(\p_{1})$
and $t_{2k}(\p_{1i})$, for $2\leq i\leq r$.  We first note the
following symmetries:
\begin{enumerate}
\renewcommand{\theenumi}{$(\arabic{enumi})$}
\renewcommand{\labelenumi}{\textbf{\theenumi}}
\item $\uu'$ and $\vv'$, and hence $\uu$ and $\vv$, are
  interchangeable (by conjugation); hence we can arbitrarily set
  $u_{1}=1$ and $v_{1}=0$;
\item concerning $\x'$, $\y'$ and $\z'$:
\begin{enumerate}
\renewcommand{\theenumi}{$(\arabic{enumi})$}
\renewcommand{\labelenumi}{\textbf{\theenumi)}}
\item if all are non-zero, and hence also distinct, then we can
  permute them arbitrarily;\label{tziris01}
\item if all are zero, then again we can permute $\x$, $\y$ and $\z$
  arbitrarily; \label{casob01}
\item otherwise, one of them is zero and the others equal and
  non-zero; we have chosen them so that $\y'=\0$ and $\x'=\z'$, so we
  can still swap $\x$ and $\z$.\label{casoc01}
\end{enumerate}
\end{enumerate}

Now $t_{2k}(\p_{1})=x_{1}y_{1}$, since $u_{1}v_{1}=0$.  Hence if
$t_{2k}(\p_{1})=1$ then we deduce that $x_{1}=y_{1}=z_{1}=1$;
prepending a $1$ to $\x'$, $\y'$ and $\z'$ gives $\x$, $\y$ and $\z$.
Otherwise, $x_{1}y_{1}=0$ and we need to determine which one of~$x_1$,
$y_1$ or~$z_1$ is~$1$, the other two being~$0$. We can compute
\begin{align*}
t_{2k}(\p_{1i})&=(u_{1}+u_{i})(v_{1}+v_{i})+(x_{1}+x_{i})(y_{1}+y_{i})\\
&=(x_{1}+x_{i})(y_{1}+y_{i})
\end{align*}
for $i\geq 2$ (using $u_1+u_i\not= v_1+v_i$) and hence get the values
$y_{1}x_{i}+x_{1}y_{i}$ for $i\geq 2$, since we already know $x_1y_1$
and all $x_iy_i$ for $i\ge2$.

Define
$$
  \qq = (t_{2k}(\p_{i})+t_{2k}(\p_{1i})+u_{i})_{i=2}^{r}
  =y_{1}\x'+x_{1}\y' \in \F_{2}^{r-1}.
$$
Consider the three cases under (2) above:
\begin{itemize}
\item In \ref{tziris01}, $\x'$ and $\y'$ are linearly independent so
  $\qq$ determines $x_{1}$ and $y_{1}$ uniquely;
\item In \ref{casob01}, we have complete symmetry and may set $\x=\y=\0$
  and $\z=\ee_{1}$;
\item In \ref{casoc01}, since $\y'=\0$ we have $\qq=y_{1}\x'$ and
  $\x'$ is not zero, so if $\qq\neq \0$ then $y_{1}=1$ and
  $x_{1}=z_{1}=0$. On the other hand, if $\qq=\0$ then $y_{1}=0$ and
  we can set $x_{1}=0$, $z_{1}=1$ (or vice versa, it does not matter
  since $\x'=\z'$).
\end{itemize}

This completes the method to determine the vectors $\uu,\vv,\x,\y,\z$
and hence the discriminants $\Delta_{a}$, $\Delta_{b}$, $\Delta_{c}$,
$\Delta_{d}$ and $\Delta_{abcd}$ and the associated characters.

In summary, when $\rho$ is trivial modulo~$2^{k}$ and has non-trivial
determinant modulo~$2^{k+1}$, we can again use the test function
values $t_{2k}(\p)$ for $\p\in T_2$ (where $T_2$ is a quadratically
independent set of primes for~$S$), together with either $t_{2k+1}$ or
one of the modified tests $\tilde t_{2k+1}$ depending on~$\p$, to
determine the full set of characters $\chi_a$, $\chi_b$, $\chi_c$,
$\chi_d$, satisfying $\chi_a+\chi_d=\chi_{\det}\not=0$, if necessary
replacing $\rho$ by a $\gl_2(\Z_2)$-equivalent representation, or by a
$2$-isogenous representation.  Unlike subsection~\ref{section1mod2k1},
it is not possible for all the characters to be trivial, and $\rho$ is
certainly not trivial modulo~$2^{k+1}$ as $\det\rho$ is
nontrivial modulo~${2^{k+1}}$.

We now summarise the results of this section.
\begin{theorem}\label{thm:largeclasses}
Let $K$ be a number field, $S$ a finite set of primes of~$K$, and
$\rho$ a $2$-dimensional $2$-adic Galois representation over~$K$
unramified outside~$S$.  Suppose that there exists a stable lattice
under the action of~$\rho$ with respect to which $\rho\pmod{2^k}$ is
trivial, for some~$k\ge1$.  Then, using the output of the Black Box
for~$\rho$ for a set~$T_2$ of primes which are quadratically
independent with respect to~$S$, we can determine whether there exists
a (possibly different) stable lattice with respect to which
$\rho\pmod{2^{k+1}}$ is trivial.  More generally we can completely
determine the representation $\rho\pmod{2^{k+1}}$ on some stable
lattice for~$\rho$.
\end{theorem}

\begin{example*}{(continuation of Example~\ref{ex:1})}
With $K=\Q$ and $S=\{2,37\}$, let $\rho$ be the Galois representation
attached to elliptic curve isogeny class \lmfdbeciso{43808}{a}, which is one of
those which in the previous section was seen to be large, indicating
that there exists an elliptic curve in the class with full rational
$2$-torsion. In fact, \lmfdbec{43808}{a}{1} is such a curve, but we stress
that the following facts about the isogeny class are being determined
from only the knowledge of the trace of Frobenius at the six primes
in~$T_2$:
\[
\begin{tabular}{c|cccccc}
$I$ & $\{1\}$& $\{2\}$& $\{3\}$& $\{1,2\}$& $\{2,3\}$& $\{1,3\}$\\
\hline
  $p_I$ & $7$ & $53$ & $17$ & $3$ & $5$ & $23$ \\
  $a_p$ & $0$ & $14$ & $-2$ & $0$ & $2$ & $2$ \\
\end{tabular}
\]

Now, $\Delta_{\det}=-1$, this being the discriminant of the cyclotomic
character on the $4$th roots of unity.  Using the method of
subsection~\ref{section12kmod2k1} with $k=1$, we compute $t_2(p_2) =
t_2(53)\equiv0$, $t_2(p_3) = t_2(17)\equiv1$, $t_2(p_{2,3}) =
t_2(5)\equiv 1$, from which $x_2y_3+x_3y_2\equiv0+1+1\equiv0$, hence
(without loss of generality) $\y'\equiv\0$ and $\x'\equiv\z'$ but the
common value is not yet known.  Write $\y_1'=(y_2',y_3')$ for the
exponent vector on $2,37$ of the discriminant of the character
denoted~$c_1$ above.  We find $x_2y_2'\equiv1$, $x_3y_3'\equiv0$ and
$(x_2+x_3)(y_2'+y_3')\equiv1$ using three computations involving the
special test functions $t_3$ and $\tilde{t_3}$ as in (\ref{senka2.1}),
(\ref{senka2.2}) and (\ref{senka2.3}) (once each).  We give details of
one of these.  Let $p= 5= p_{2,3}$, for which the trace is (exactly)
$a_5=2$.  Now $u_2\equiv t_2(p_2)\equiv t_2(53)\equiv 0$ and
$u_3\equiv t_2(p_3)\equiv t_2(17)\equiv 1$, so $u_2+u_3\equiv1$ and
hence we are in the case where $a$ and~$d$ are both odd with
$ad\equiv-1\pmod4$, so $\tilde{t}_3(5)=((1+5-a_5+4)/8)\equiv 1$; this
implies that~$(x_2+x_3)(y_2'+y_3')\equiv1\pmod2$.  The two similar
computations use $a_{53}=14$ and $a_{17}=-2$ to obtain $t_3(53)\equiv
1$ and $\tilde{t}_{3}(17)\equiv 0$.

Solving the congruences for $x_2,x_3,y_2',y_3'$ we find
$\x'\equiv \y_1'\equiv (1,0)$.

Next, $\qq\equiv (t_2(p_i)+t_2(p_{1,i})+u_i)_{i=2}^{3}=(1,0)\not\equiv
\0$, so $y_1\equiv 1$ and $x_1\equiv z_1\equiv 0$.  Finally we have
$\x\equiv \z\equiv (0,1,0)$, $\y\equiv (1,0,0)$, and so $\Delta_b=
\Delta_{abcd}=2$ while $\Delta_c=-1$.  Also $\uu=(1,0,1)$ and
$\vv=(0,0,1)$, so $\Delta_a=-37$ and $\Delta_d=+37$.  The image of the
mod~$4$ representation has order~$2^3=8$ since the space spanned
by~$\x,\y,\z,\uu,\vv$ has dimension~$3$, and its kernel has fixed
field~$\Q(\sqrt{-1},\sqrt{2},\sqrt{37})$.

To confirm this, the three elliptic curves $2$-isogenous to
\lmfdbec{43808}{a}{1} do indeed have discriminants which are square
multiples of~$-1$, $2$ and~$2$.
\end{example*}

\section{Detecting triviality of the semisimplification}\label{sec:triviality}
In the past three sections we have given algorithms for determining
the following properties of a continuous $2$-dimensional $2$-adic
Galois representation~$\rho$, unramified outside a given finite set of
primes~$S$, using only the output from a Black Box oracle giving for
any prime~$\p\notin S$ the Frobenius polynomial~$F_\p(t)$:
\begin{enumerate}
\item whether or not~$\rho$ is residually reducible
  (Theorem~\ref{thm:reducibleresidual}: using the primes in a distinguishing set~$T_0$ for~$S$);
\item if $\rho$ is residually reducible, whether or not~$\rho$ is
  residually trivial up to isogeny (Theorem~\ref{thm:trivialmod2}:
  using the primes in a quadratically independent set~$T_2$ with
  respect to~$S$);
\item if $\rho$ is trivial modulo~$2^k$ up to isogeny, whether or
  not~$\rho$ is trivial modulo~$2^{k+1}$ up to isogeny
  (Theorem~\ref{thm:largeclasses}: again using the primes in a
  quadratically independent set~$T_2$).
\end{enumerate}
We also showed in Section~\ref{sec:characters} how to verify
that~$\det\rho$ was equal to a given $2$-adic character
(Theorem~\ref{determinante}, using the primes in a linearly
independent set~$T_1$ with respect to~$S$).

So far we have only needed finite $2$-adic precision from our Black
Box oracle.  In this section we assume that the oracle can provide us
with the Frobenius polynomials~$F_\p(t)$ exactly, which is usually the
case in practice when they are monic polynomials in~$\Z[t]$.  By
putting together the previous results we can determine whether $\rho$
has trivial semisimplification; since we only know $\rho$ through the
characteristic polynomials of the $\rho(\sigma)$, this is as close as
we can get to showing that $\rho$ is trivial.

We start with a lemma taken from the proof of
Theorem~\ref{determinante}:

\begin{lemma}\label{propodet1}
Let $\chi\colon G_{K}\to\Z_{2}^{*}$ be a continuous character unramified outside $S$. If
\begin{enumerate}
\item $\chi(\sigma)\equiv 1\pmod {2^{k-1}}$ for all $\sigma\in G_{K}$, and
\item $\chi(\frob{\p})\equiv 1\pmod {2^{k}}$ for all $\p\in T_{1}$,
\end{enumerate}
where $T_{1}$ is a linearly independent set with respect to~$S$, then
$\chi(\sigma)\equiv 1\pmod {2^{k}}$ for all $\sigma\in G_{K}$.
\end{lemma}

\begin{proposition}\label{theorem3.4.3}
Let $\rho\colon G_{K}\rightarrow \gl_{2}(\Z_{2})$ be a Galois representation unramified outside $S$ such that
\begin{center}
$\rho(\sigma)\equiv \II\pmod{2^{k}}$ for all $\sigma\in G_{K}$.
\end{center}
Suppose that
\begin{enumerate}
\item $\det(\rho(\frob{\p}))\equiv 1\pmod{2^{k+1}}$ for all $\p\in
  T_{1}$, and
\item $F_\p(1) \equiv 0 \pmod{2^{2k+2}}$ for all $\p\in T_{2}$,
\end{enumerate}
where $T_{1}$ is a linearly independent set and $T_{2}$ a
quadratically independent set with respect to~$S$.  Then there exists
an isogenous representation $\rho'$ such that $\rho'(\sigma)\equiv
\II\pmod{2^{k+1}}$ for all $\sigma\in G_{K}$.
\end{proposition}

\begin{proof}
First, by Lemma \ref{propodet1}, the first condition implies that
$\det(\rho(\sigma))\equiv1\pmod{2^{k+1}}$ for all $\sigma\in G_{K}$.

Next we use the notation of the previous section,
specifically~(\ref{rhomod2k}).  The determinant condition just
established shows that $a+d\equiv0\pmod2$ and we are in the case
$\Delta_{\det}=1$ as in subsection~\ref{section1mod2k1} with
$\uu=\vv$.  Next, $F_\p(1)\equiv0\pmod{2^{2k+2}}$ means that all the
test function values are~$0$.  This gives in turn $\WW=\0$, $\y=\0$
and $\uu=\vv=\0$.  Finally we have $bc_1\equiv0\pmod2$ so (applying a
$2$-isogeny if necessary) we may assume that~$b\equiv0$, so $\x=\0$.  Hence all
the characters are trivial, as required.
\end{proof}

Using this proposition, we can prove our final result.
\begin{theorem}
Let $\rho\colon G_{K}\rightarrow\gl_{2}(\Z_{2})$ be a continuous
Galois representation unramified outside $S$ which is residually
reducible.  If
\begin{enumerate}
\item $\det(\rho(\frob{\p}))=1$ for all $\p\in T_{1}$, and \label{001}
\item $\tr(\rho(\frob{\p}))=2$ for all $\p\in T_{2}$, \label{002}
\end{enumerate}
(in particular, if $\frobp$ has characteristic polynomial $(t-1)^2$
for all~$\p\in T_2$), then $\rho$ is reducible, with trivial
semisimplification, and is of the form
\[
  \rho(\sigma) = \begin{pmatrix} 1 & * \\ 0 & 1
  \end{pmatrix}
\]
with respect to a suitable basis.
\end{theorem}
\begin{proof}
Suppose that $\rho$ were irreducible; then $\BT(\rho)$ is finite, and
none of the finitely many integral forms $\rho_\Lambda$ is trivial
(otherwise $\rho$ would be) so there is a maximal $k\ge1$ such that
$\rho_{\Lambda}$ is trivial modulo~$2^k$ for some stable
lattice~$\Lambda$.  This contradicts Proposition~\ref{theorem3.4.3}.
Hence $\rho$ is reducible.

With respect to a suitable basis all the matrices $\rho(\sigma)$ are
upper triangular. The diagonal entries determine characters of~$G_K$,
which are both trivial on $\frobp$ for all~$\p\in T_1$ (since the
product of their values is~$1$ and their sum~$2$).  By
Theorem~\ref{determinante} both diagonal characters are trivial.
\end{proof}

\section{Further examples}
We finish by presenting two examples with base field
$K=\Q(\sqrt{-1})$, where the Black Box Galois representations come
from Bianchi modular newforms with rational Hecke eigenvalues.  The
existence of suitable Galois representations in this case was first
developed by Taylor \textit{et al.} in \cite{TaylorI}, \cite{TaylorII}
with subsequent results by Berger and Harcos in~\cite{Berger}.  For
our purposes we only need the existence of the representation and the
knowledge that it is unramified outside the primes dividing the level
of the newform, with the determinant and trace of Frobenius at an
unramified prime $\p$ equal to the norm~$N(\p)$ and the Hecke
eigenvalue~$a_{\p}$ respectively.  These eigenvalues were computed in
these examples using the methods of \cite{CremonaHyp}.  The newforms
we use here are in the LMFDB~\cite{lmfdb} and may be found at
\url{http://www.lmfdb.org/ModularForm/GL2/ImaginaryQuadratic/}.

In both these examples (as in several hundred thousand others we have)
there exist elliptic curves defined over~$K$ whose $2$-adic Galois
representation can be proved to be equivalent to the representation
attached to the newform, using the Serre-Faltings-Livn\'e method as
detailed in~\cite{pacetti}.  However in preparing the examples we did
not use the elliptic curves themselves, but used modular symbol
methods to obtain the traces of Frobenius as Hecke eigenvalues.  As
$\det(\rho(\frobp))=N(\p)$ and we include the prime above~$2$ in~$S$,
for $K=\Q(\sqrt{-1})$ we always have $N(\p)\equiv1\pmod{4}$, and hence
the determinant of the representation is trivial modulo~$4$.

In this way we can obtain information about the elliptic curves
conjecturally associated to a rational Bianchi newform, even in cases
where we have not been able to find a suitable elliptic curve.

\begin{example}\label{exampleN3140c}
The base field is $K=\Q(i)$, where we write $i=\sqrt{-1}$.  The Galois
representation we consider is that attached to the Bianchi newform
with LMFDB label \lmfdbbmf{2.0.4.1}{3140.3}{c}.
Here, 2.0.4.1 is the LMFDB label for~$K$; then 3140.3 is the label for
the level $\N=(56+2i)=(1+i)^2(1+2i)(11+6i)$ of norm~$3140$ (it is the
$3$rd ideal of this norm in the ordering used by the LMFDB).  Finally
the suffix c identifies the newform itself: the new space at
level~$\Gamma_0(\N)$ is three-dimensional with a basis of three
newforms (labelled a, b and c) each with rational Hecke eigenvalues.

Let $S=\{1+i,1+2i,11+6i\}$ be the set of primes dividing the level,
outside which the representation is unramified. Then
\[
  K(S,2)=\left\langle 1+i,1+2i,11+6i,i\right\rangle\cong (\Z/2\Z)^{4}.
\]
There is one $C_3$ extension of $K$ unramified outside~$S$, and $5$
$S_3$ extensions, so we have a set~$\FF$ of $6$ possible cubics.
Using Algorithm~\ref{algoritmoT0} we find that a suitable distinguishing
set is $T_0=\{2+i, 2+3i, 3+2i, 1+4i\}$.  Checking that $a_\p$ is even
for all~$\p\in T_0$ shows that the mod-$2$ representation is
reducible.

Using Algorithm~\ref{algoritmoT2*} we find the following set of ten
primes forms a special quadratically independent set.  (We only use
primes of degree~$1$ here, noting that the cost of computing $a_\p$
grows with $N(\p)$.)
\[
\setlength{\tabcolsep}{1pt} 
\begin{tabular}{c|cccccccccc}
$I$ & $\{1\}$& $\{2\}$& $\{3\}$& $\{4\}$& $\{1,2\}$& $\{1,3\}$& $\{1,4\}$& $\{2,3\}$& $\{2,4\}$& $\{3,4\}$\\
\hline
  $\p_I$ & $(1+4i)$ & $(4+5i)$ & $(8+7i)$ & $(5+2i)$ & $(4+i)$ & $(5+8i)$ & $(6+i)$ & $(5+4i)$&  $(2+i)$ & $(2+3i)$ \\
  $a_\p$ & $2$ & $10$ & $10$ & $6$ & $2$ & $-14$& $-2$ & $-2$ & $2$ & $-6$ \\
  $F_\p(1)$ & $16$ & $32$ & $104$ & $24$ & $16$ & $104$& $40$ & $44$ & $4$ & $20$ \\
  $t_1(\p)$ & $0$ & $0$ & $0$ & $0$ & $0$ & $0$& $0$ & $0$ & $0$ & $0$ \\
  $t_2(\p)$ & $0$ & $0$ & $0$ & $0$ & $0$ & $0$& $0$ & $1$ & $1$ & $1$ \\
\end{tabular}
\]
Applying the test $t_{1}(\p)$, given by $(\ref{test1})$, amounts to
testing whether each~$a_\p\equiv0$ or~$2\pmod4$; here,
all~$a_\p\equiv2\pmod4$.  (In the notation of
subsection~\ref{subsec:quadindep}, we have $\vv=\0$.) This implies
that the width of the isogeny class is at least $2$; we have a large
isogeny class.

To determine whether the width of the isogeny class is actually $2$,
we apply the test $t_{2}(\p)$, given by $(\ref{t2k1det1})$ (since the
determinant is trivial mod~$4$) with $k=1$.  Using the method of
subsection~\ref{section1mod2k1} we construct the matrix
\begin{align*}
\WW = \begin{pmatrix} 0&0&0&0\\0&0&1&1\\0&1&0&1\\0&1&1&0\end{pmatrix}
\end{align*}
which has rank~$2$, so (without loss of generality) we take
$\x=(0,0,1,1)$, $\y=(0,1,0,1)$ (two distinct non-zero rows of~$\WW$)
and $\z=\x+\y=(0,1,1,0)$. Then from (\ref{testfunctionst2kmod1k}) we find that
$\uu=\vv=(0,0,0,1)$. Therefore $\Delta_{b}=(11+6i)(i)$,
$\Delta_{c}=(1+2i)(i)$, $\Delta_{abcd}=(11+6i)(1+2i)$.  So in this
case the stable Bruhat-Tits tree (see Figure~\ref{width2tree})
has four vertices, with discriminants as shown here:
\[
\includegraphics[scale=0.5]{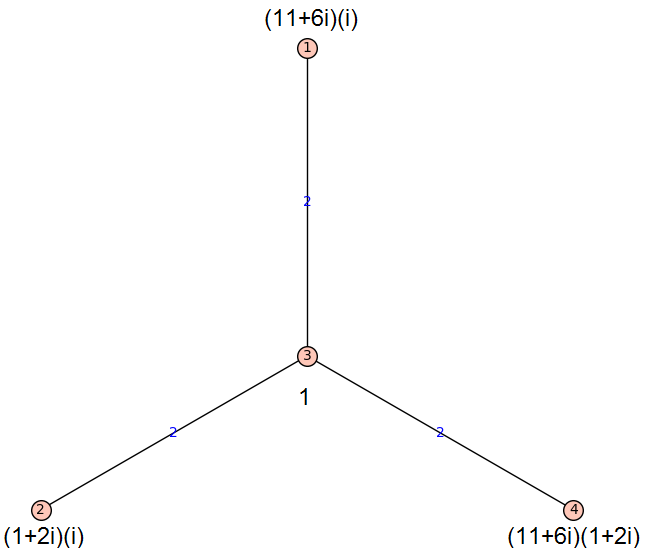}
\]
We can match the data presented in this example to the isogeny class
\lmfdbecnfiso{2.0.4.1}{3140.3}{c} of elliptic curves of conductor $\N$ over $\Q(i)$.
The four elliptic curves in this class include one (with label c3)
with full $K$-rational $2$-torsion, while each of the others (labelled
c1,c2,c4) has a single $K$-rational point of order~$2$ as expected.
Moreover the discriminants of the other three curves are (up to square
factors) $i(11+6i)$, $i(1+2i)$ and~$(1+2i)(11+6i)$ respectively.
\end{example}

\begin{example}
Again with $K=\Q(i)$ as base field, let $S=\{1+i,2+i,2-i\}$.  For this
example we take the Bianchi newform with LMFDB label \lmfdbbmf{2.0.4.1}{200.2}{a}
with level $\N=(10+10i)=(1+i)^3(2+i)(2-i)$ of norm~$200$,
a base-change of a classical newform on $\Gamma_0(40)$.  Now
\[
  K(S,2)=\left\langle i, i + 1, -i - 2, 2i + 1\right\rangle\cong(\Z/2\Z)^{4};
\]
we have put the unit~$i$ first since we will be using the method of
subsection~\ref{section12kmod2k1}.  Now~$\FF$ has only one element and
$T_0=\{4+i\}$.  Since $a_{4+i}=2$ the representation is residually
reducible.

We find $T_2$ as before and obtain the following data from the
newform, acting as our Black Box:
\[
\setlength{\tabcolsep}{1pt} 
\begin{tabular}{c|cccccccccc}
$I$ & $\{1\}$& $\{2\}$& $\{3\}$& $\{4\}$& $\{1,2\}$& $\{1,3\}$& $\{1,4\}$& $\{2,3\}$& $\{2,4\}$& $\{3,4\}$\\
\hline
  $\p$ & $(2+3i)$ & $(5+8i)$ & $(8+7i)$ & $(7+8i)$ & $(6+i)$ & $(5+2i)$ & $(6+5i)$ & $(1+4i)$&  $(4+i)$ & $(4+5i)$ \\
  $N(\p)$ & $13$ & $89$ & $113$ & $113$ & $37$ & $29$ & $61$ & $17$&  $17$ & $41$ \\
  $a_\p$ & $-2$ & $-6$ & $18$ & $18$ & $6$ & $-2$& $-2$ & $2$ & $2$ & $-6$ \\
  $F_\p(1)$ & $16$ & $96$ & $96$ & $96$ & $32$ & $32$& $64$ & $16$ & $16$ & $48$ \\
  $t_4(\p)$ & $1$ & $0$ & $0$ & $0$ & $0$ & $0$& $0$ & $1$ & $1$ & $1$ \\
\end{tabular}
\]
Since $t_k(\p)=0$ for all $\p\in T_2$ for $k=1,2,3$ we see that not
only is $\rho$ residually reducible, it is even trivial modulo~$4$ (up
to isogeny).  Fixing a stable lattice with respect to which $\rho$ is
trivial mod~$4$, we will determine $\rho\pmod8$, noting that it does
not have trivial determinant, as some primes have
norm~$\not\equiv1\pmod8$.

Using the method of subsection~\ref{section12kmod2k1}, we evaluate
$t_4(\p)$ for $\p\in T_2$ (see the last row of the table above).  From
this we evaluate the $3\times3$ matrix
\begin{align*}
\emph{\textbf{W}}'=\begin{pmatrix}0&1&1\\1&0&1\\ 1&1&0 \end{pmatrix}
\end{align*}
and observe that it has rank~$2$.  Thus we may take $\x'=(0,1,1)$,
$\y'=(1,0,1)$ and $\z'=(1,1,0)$. From $t_{4}(\p_{1})=1$ we then get
$x_{1}=y_{1}=z_{1}=1$, so $\x=(1,0,1,1)$, $\y=(1,1,0,1)$,
and~$\z=(1,1,1,0)$.  Using (\ref{testfunctionst2kmod2k}) gives
$\uu'=\vv'=(0,0,1)$ so $\uu=(1,0,0,1)$, $\vv=(0,0,0,1)$, completing
the determination of $\rho\pmod8$: it is the full subgroup of
$\gl_2(\Z/8\Z)$ consisting of matrices congruent to the identity
modulo~$4$. Since all of $\x$, $\y$, $\z$ are non-zero, we have also
determined the entire isogeny class (stable Bruhat-Tits tree), and
know that the discriminants of its leaves are (modulo squares)
$\Delta_{b}=i(2+i)(1+2i)$, $\Delta_{c}=i(1+i)(1+2i)$ and
$\Delta_{abcd}=i(1+i)(2+i)$.  The isogeny graph therefore looks like
this:
\[
\includegraphics[scale=0.49]{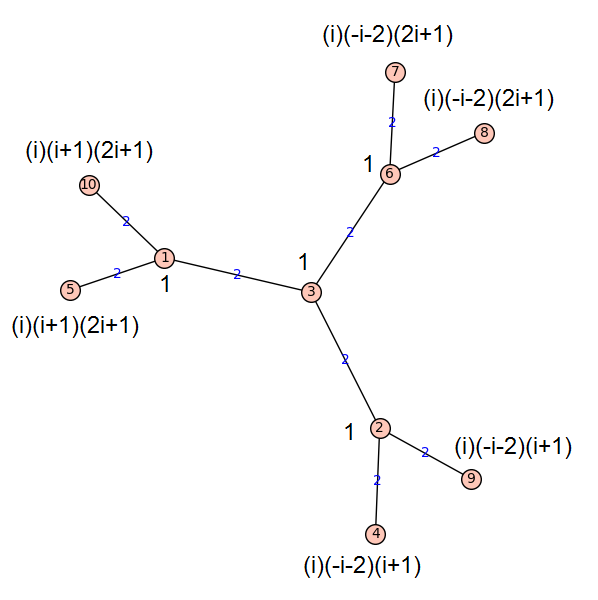}
\]
Moreover, we also have (using the notation of~\ref{discriminants})
$\Delta_{a}=i(1+2i)$ and $\Delta_d=1+2i$, so by
(\ref{rhomod2k}) we have $\rho(\sigma) \equiv
I+4\begin{pmatrix} a(\sigma) &
b(\sigma)\\ c(\sigma)&d(\sigma) \end{pmatrix} \pmod8$, where
$a(\sigma)\equiv0\pmod2\iff\sigma(\sqrt{\Delta_a})=+\sqrt{\Delta_a}$,
and similarly for~$b,c,d$.

We can match the data presented in this example to the isogeny class
\lmfdbecnfiso{2.0.4.1}{200.2}{a} of elliptic curves of conductor $\N$ over $\Q(i)$.
The class includes the base-change to~$K$ of isogeny class 40a of
elliptic curves defined over~$\Q$ with conductor~$40$; it consists of
ten curves a1--a10 and the graph of $2$-isogenies between them is as
in this diagram, where the vertex labels match the number labels of
the curves in the class.  Consulting the LMFDB page cited, we can
check that the discriminants of the ten curves are as indicated in the
diagram (up to squares).
\end{example}

\subsection*{Acknowledgments}
This work formed part of the 2016 PhD thesis \cite{argaez01} of the
first author, during which he was supported by a Chancellor's
Scholarship from the University of Warwick.  The second author is
supported by EPSRC Programme Grant EP/K034383/1 \textit{LMF: L-Functions and
Modular Forms}, and the Horizon 2020 European Research
Infrastructures project \textit{OpenDreamKit} (\#676541).

\bibliographystyle{plain} 


\end{document}